\documentclass[12pt]{article}

\usepackage[latin1]{inputenc}
\usepackage{epsfig}
\usepackage{graphicx,psfrag}
\usepackage[tight]{subfigure}

\usepackage{amsfonts,amsthm,bbm,amssymb,epsf,amsmath,%
latexsym,amscd,amsfonts,enumerate,amsthm,supertabular,color,url,multirow}
\usepackage[colorlinks,bookmarksopen,bookmarksnumbered,citecolor=blue,urlcolor=blue]{hyperref}

\newtheorem{theorem}{Theorem}[subsection]
\newtheorem{lemma}[theorem]{Lemma}
\newtheorem{corollary}[theorem]{Corollary}
\newtheorem{proposition}[theorem]{Proposition}

\newtheorem*{keywords}{Keywords}
\newtheorem{question}{Question}[section]

	\let\Item\item

\def\address#1{\expandafter\def\expandafter\@aabuffer\expandafter
	{\@aabuffer{\affiliationfont{#1}}\relax\par
	\vspace*{13pt}}}

\numberwithin{equation}{section}

\begin{document}

\date{}

\title{Irreducible representations of certain nilpotent groups of finite rank}

\author{ Anatolii V. Tushev 
\thanks{The author has received funding through the MSCA4Ukraine project, which is funded by the European Union (ID 1232926)} \\
        Justus Liebig University Giessen\\
        Giessen 35390, Germany;\\
        {\it E-mail: Anatolii.Tushev@math.uni-giessen.de}\\
       }

\maketitle

\begin{abstract}
In the paper we study irreducible representations of some nilpotent groups of finite abelian total rank. The main result of the paper states that if  a torsion-free minimax group $G$ of nilpotency class 2 admits a faithful irreducible representation $\varphi $ over a finitely generated field $k$ such that $char \, k \notin Sp(G)$ then there exist a subgroup $N$ and an irreducible primitive representation $\psi $ of the subgroup $N$ over $k$ such that the representation $\varphi $ is induced from $\psi $ and the quotient group $N/Ker\psi $ is finitely generated. 
 \begin{keywords}
abelian total rank, primitive representations
\end{keywords}
\end{abstract}

\section{Introduction} \label{Section 1}

          A group $G$ is said to have finite (Prufer) rank if there is a positive integer $m$ such that any finitely generated subgroup of $G$ may be generated by $m$ elements; the smallest $m$ with this property is the rank $r(G)$ of $G$. A group $G$ is said to be of finite torsion-free rank if it has a finite series each of whose factor is either infinite cyclic or locally finite; the number ${r_0}(G)$ of infinite cyclic factors in such a series is the torsion-free rank of $G$. \par 
If a group $G$ has a finite series each of whose factor is either cyclic or quasi-cyclic then $G$ is said to be minimax; the number $m(G)$ of infinite factors in such a series is the minimax length of $G$. If in such a series all infinite factors are cyclic then the group $G$ is said to be polycyclic; the number $h(G)$ of infinite factors in such a series is the Hirsch number of $G$.  \par 

Let $G$ be an abelian group and $ t(G) $ be the torsion subgroup of $ G $. Let $p \in \pi(t(G))$  and $G_p$ be the 
Sylow $ p $-subgroup of $t(G)$, where $\pi(t(G))$ is the set of prime divisors of orders of elements of $ t(G) $. 
Then we can define the total rank ${r_t}(G)$ of $G$ by the following formula: ${r_t}(G) = r(G/t(G)) + \sum\nolimits_{p \in \pi(t(G)) } r(G_p) $. \par
A soluble group has finite abelian total rank, or is a soluble FATR-group, if it has a finite series in which each factor is abelian of finite total rank. Many results on the construction of soluble FATR-groups can be found in \cite{LeRo2004}. \par 

Let $G$ be a group, let $R$ be a ring and let $I$ be a right ideal of the group ring $RG$. We say that a subgroup $H$ of the group $G$ controls the ideal $I$ if 

\begin{equation}
                                                             I = (I \cap RH)RG.                 \label{1.1}
\end{equation}

\par  
	Let $H$ be a subgroup of the group $G$ and let $U$ be a right $RH$-module. Since the group ring $RG$ can be considered as a left $RH$-module, we can define the tensor product $U{ \otimes _{RH}}RG$, which is a right $RG$-module named the $RG$-module induced from the $RH$-module $U$. Moreover, if $M$ is an $RG$-module and $U \le M$ then it follows from \cite[Chap. 2, Corollary 1.2(i)]{Karp} that

\begin{equation}
                                                 M = U{ \otimes _{RH}}RG                            \label{1.2}
\end{equation}

if and only if 

\begin{equation}
                                                  M = { \oplus _{t \in T}}Ut,                                     \label{1.3}
\end{equation}

where $T$ is a right transversal to subgroup $H$ in $G$. If the equality \ref{1.2} holds we say that the subgroup $H$ controls the module $M$. 
 \par 
	Suppose that $M = aRG$ is a cyclic right $RG$-module generated by a nonzero element $a \in M$. Put $I = An{n_{kG}}(a)$ and let $U = akH$, where $H$ is a subgroup of the group $G$. It is not difficult to note that, in these denotations, the equality \ref{1.1} holds if and only if the equality \ref{1.3} holds. Thus, in the case $M = aRG$, the equalities \ref{1.1}, \ref{1.2}, \ref{1.3} mean the same. 
\par  
The equality \ref{1.3} shows that properties of the $RG$-module $M$ and the $RH$-module $U$ are closely related. For instance , if the module $M$ has some  chain condition (for example, is Artenian or Noetherian) then the module U also has this condition. So, the equality \ref{1.2} may be very useful if properties of the $RH$-module $U$ are well studied. For instance, in the case where the group $H$ is polycyclic, we have a deeply developed theory (see \cite{Wehr09}). The equality \ref{1.2} also may be very useful in the case where the group $G$ has finite torsion-free rank ${r_0}(G)$ if ${r_0}(H) < {r_0}(G)$ because we can use the induction on ${r_0}(G)$ then. However, on the place of ${r_0}(G)$ there may be another rank of the group $G$ or the minimax length $m(G)$ if $G$ is minimax. 
\par    
If a $kG$-module $M$ of some representation  $\varphi $ of a group $G$ over a field $k$ is induced from some  $kH$-module $U$, where $H$ is a subgroup of the group $G$, then we say that the representation $\varphi $ is induced from a representation $\phi $ of subgroup $H$, where $U$ is the module of the representation $\phi $. Recall that the representation $\varphi $ is said to be faithful if $Ker\varphi  = 1$, it is equal to ${C_G}(M) = 1$. 
\par  

In Section 3 we managed to characterize faithful irreducible representations of an abelian group $G$ of finite total rank as representations induced from representations of finitely generated subgroups of $G$ (Corollary 3.1.5). The proof of this result is based on properties of the multiplicative group of fields (Proposition 3.1.1) and generalizations (Proposition 3.1.2) of some results of Kummer theory (see \cite{Lang1965}, Chap. VIII, §8).\par  

Let $R$ be a ring, $G$ be a group and $I$ be a right (left) ideal of the group ring $RG$. It is not difficult to show that ${I^\dag } = (1 + I) \cap G$ is a subgroup of $G$ and if the ideal $I$ is  two-sided then ${I^\dag }$ is a normal subgroup of $G$. The ideal $I$ is said to be faithful if ${I^\dag } = 1$. The approach based on Propositions 3.1.1 and 3.1.2 also allows us to study the properties of faithful prime ideals of group rings of abelian groups. In \cite{Sega2001} Segal proved that under some additional conditions  any faithful prime ideal of a group ring $RG$ of an abelian minimax group $G$ over a finitely generated commutative ring $R$ is finitely generated. In particular, we obtained a criterion when all faithful prime ideals of a group algebra $kG$ of an abelian group $G$ over a finitely generated field $k$ are finitely generated (Theorem 3.2.2). \par 
     Let $G$ be a group , let $k$ be a field and let $M$ be a $kG$ -module. The module $M$ is said to be primitive if it is not induced from any $kH$-submodule for any subgroup $H < G$.  The module $M$ is said to be semiprimitive if it is not induced from any $kH$-submodule for any subgroup $H < G$ such that $\left| {G:H} \right| < \infty $. A representation  $\varphi $ of $G$ over $k$ is said to be primitive (semiprimitive) if the module of the representation $\varphi $. \par
In Section 5 we consider the structure of nilpotent FATR-group of nilpotency class 2 which admits faithful semiprimitive irreducible representations over a finitely generated field $k$ such that $char\,k \notin Sp(G)$ (Theorem 5.1.5). The obtained result allows us to describe faithful irreducible representations of nilpotent minimax groups of nilpotency class 2 (see Theorem 5.2.2 and Corollary 5.2.3). These results raise the following question. \par  
 
\begin{question}\label{question 1.1}

Let $G$ be a torsion-free minimax nilpotent group and let $k$ be a finitely generated field such that $char\,k \notin Sp(G)$. Suppose that the group  $G$ admits a faithful irreducible representation $\varphi $ over $k$. Is it true that then there exist a subgroup $N$ of $G$ and an irreducible primitive representation $\psi $ of the subgroup $N$ over $k$ such that the representation $\varphi $ is induced from $\psi $ and the quotient group $N/Ker\psi $ is finitely generated ?

\end{question}

\section{Affirmative results} \label{Section 2}

\subsection{Group-theoretic lemmas} \label{Subsection 2.1}

         \begin{lemma}\label{Lemma 2.1.1} Let $G$ be an abelian group and let $T$ be the torsion subgroup of $G$. Let $H$ be a finitely generated subgroup of $G$. Then: \par
          (i)  if the set $\pi (T)$ is infinite then there is a nontrivial subgroup $X$ of $T$ such that $X \cap H = 1$; \par
           (ii) if the quotient group $G/T$ has infinite rank then there is a free abelian subgroup $Y \le G$ of infinite rank such that $Y \cap H = 1$.       
\end{lemma}
         \begin{proof} (i). Let $S$ be the torsion subgroup of $H$. By \cite[Theorem 15.5]{Fuch1973}, $H$ satisfies the condition of maximality for subgroups and it easily implies that the set $\pi (S)$ is finite. Then we can put $X = { \times _{p \in (\pi (T)\backslash \pi (S))}}{T_p}$, where ${T_p}$ is a $p$-component of $T$. \par
         (ii).  Since the torsion-free quotient group $G/T$ has infinite rank, the torsion-free rank ${r_0}(G)$ of $G$ is infinite. Put $D = i{s_G}(H)$, it follows from \cite[Theorem 15.5]{Fuch1973} that ${r_0}(H)$ is finite and hence, as the quotient group $D/H$ is torsion, we can conclude that ${r_0}(D)$ is finite. Therefore, ${r_0}(G/D)$ is infinite and hence there is a subgroup $U$ of $G$ such that the quotient group $U/D$ is free abelian of infinite rank. Then it follows from \cite[Theorem 14.4]{Fuch1973} that $U = Y \oplus D$ and hence $Y$ is a free abelian subgroup of infinite rank of $G$ such that $Y \cap H = 1$. 
\end{proof}

If $G$ is a group then the $FC$-center $\Delta (G) = \{ g \in G|\left| {G:{C_G}(g)} \right| < \infty \} $ of $G$ is a characteristic subgroup of $G$. \par
Let $H$ be a subgroup of a group $G$, the subgroup $H$ is said to be dense in $G$ if for any $g \in G$ there is an integer $ n \in \mathbb{N} $ such that ${g^n} \in H$. If ${g^n} \in G\backslash H$ for any  $ n \in \mathbb{N} $ and any $g \in G\backslash H$ then the subgroup $H$ is said to be isolated in $G$. If the group $G$ is locally nilpotent then the isolator  $is_{G}(H) = \{ g \in G|{g^n} \in H\, for\, some\, n \in \mathbb{N}\}$   of $H$ in $G$ is a subgroup of $G$ 
and if $H$  is a normal subgroup then so is $is_{G}(H)$. \par

         \begin{lemma}\label{Lemma 2.1.2} Let $G$ be a nilpotent group of nilpotency class 2 and $Z$ be the centre of $G$. Then: \par
         (i) $[x,yz] = [x,y][x,z]$ and $[xy,z] = [x,z][y,z]$ for all $x,y,z \in G$;\par
         (ii) the commutator map ${\varphi _a}:G \to Z$ given by ${\varphi _a}:x \mapsto [a,x]$ is a homomorphism for any $a \in G$; \par    
         (iii) if $Y$ is a subgroup of $G$ whose normalizer ${N_G}(Y)$ has finite index in $G$ then there are a positive integer $ n \in \mathbb{N} $  and a normal subgroup $\bar Y \le G$ such that ${Y^n} \le \bar Y \le Y$;\par
	(iv) $\Delta (G) \le i{s_G}(Z)$.      
\end{lemma}

          \begin{proof} (i) The assertion follows from well-known commutator identities $[x,yz] = [x,z]{[x,y]^z}$ and $[xy,z] = {[x,z]^y}[y,z]$ because $G$ centralizes the commutant $[G,G]$. \par
           (ii)  The assertion follows from (i).  \par         
           (iii)  Since  $|G:{N_G}(Y)|<\infty$, there is $n \in \mathbb{N}$  such that ${x^n} \in {N_G}(Y)$ for all elements $x \in G$ and hence $[{x^n},y] \in Y$ for all elements $y \in Y$.  It follows from (i) that  $[{x^n},y] = {[x,y]^n} = [x,{y^n}]$ for all $x,y \in G$ and any $n \in \mathbb{N}$. Therefore, $[x,{y^n}] \in Y$ for all elements $x \in G$ and $y \in Y$. It implies that ${({y^n})^x} \in Y$ for all $x \in G$ and $y \in Y$. So, we see that ${({Y^n})^x} \le Y$ for all $x \in G$ and hence $\bar Y = \left\langle {{{({Y^n})}^x}|x \in G} \right\rangle  \le Y$. Evidently, $\bar Y$ is a normal subgroup of $G$ and ${Y^n} \le \bar Y \le Y$. \par     
           (iv) Let $x \in \Delta (G)$ then $\left| {G:{C_G}(x)} \right| < \infty $ and hence there is $n \in \mathbb{N}$ such that ${G^n} \le {C_G}(x)$, i.e. $[x,{g^n}] = 1$ for any $g \in G$. Then it follows from (i) that $1 = [x,{g^n}] = {[x,g]^n} = [{x^n},g]$ for any $g \in G$ and hence ${x^n} \in Z$. Therefore, $x \in i{s_G}(Z)$ and hence $\Delta (G) \le i{s_G}(Z)$. 
\end{proof}

Let $A$ and $H$ be subgroups of a group ${\rm{ G}}$. We say that $A$ is an $H$-invariant subgroup if $H \le {\rm{ }}{{\rm{N}}_{\rm{G}}}{\rm{ (A)}}$. 

\begin{lemma}\label{Lemma 2.1.3} Let $N$ be a nilpotent normal subgroup of finite rank of a group $G$ and let $H$ be a finitely generated dense $N$-invariant subgroup of $N$ whose derived subgroup $H'$ is $G$-invariant. Let $L = \bigcap\nolimits_{g \in G} {{H^g}} $ and suppose that $\left| {H:L} \right| = \infty $. Then there are a countable subset $\{g_i \in G\, |\, i \in \mathbb{N}\}$ and a descending chain $\{H_i \, |\, i \in \mathbb{N}\}$ of $H$-invariant subgroups ${H_i} \le H$ of finite index in $H$ such that the quotient group $H/(\bigcap_{i\in \mathbb{N}} H_i)$ is free abelian and the subgroups of $\{H_i \, |\, i \in \mathbb{N}\}$ have the following properties:\par
	(i) either there is a prime integer $p \in \pi (N/H)$ such that $H/{H_i}$ is a $p$-group for all  i $\in \mathbb{N}$ or for any prime integer $p$ there is  $n \in \mathbb{N}$ such that $p \notin \pi ({H_n}/{H_i})$ for all $i \ge n$;\par
(ii) $\bigcap\nolimits_{j = 1}^i {(H \cap {H^{{g_j}}})}  \subseteq {H_i}$ for any $i \in \mathbb{N}$;\par  
(iii) $\pi (H/{H_i}\,) \subseteq \pi (N/H\,)$ for any $i \in \mathbb{N}$. 
\end{lemma}

       \begin{proof} (i) Since the quotient group $H/L$ is countable, there is a countable subset $\{g_i \in G\, |\, i \in \mathbb{N}\}$ such that $ \bigcap_{i\in \mathbb{N}} (H\cap H^{g_i})=L $. Then $ \{L_i\, | \, L_i = \bigcap_{j=1}^i (H\cap H^{g_i})\}$ is a descending chain of subgroups such that $ \bigcap_{i\in \mathbb{N}} L_i=L $. By \cite[Proposition 2.3]{Tush2022}, there is there is a descending chain  of $H$-invariant subgroups ${H_i} \le H$ of finite index in $H$ such that the quotient group  $ H/(\bigcap_{i \in \mathbb{N}}H_i)$ is free abelian, ${L_i} \le {H_i}$ for any $i \in \mathbb{N}$ and either there is a prime integer $p \in \pi (N/H)$ such that $H/{H_i}$ is a $p$-group for all $i \in \mathbb{N}$ or for any prime integer $p$  there is $n \in \mathbb{N}$ such that  $p \notin \pi ({H_n}/{H_i})$ for all $i \ge n$.\par
(ii) Since ${L_i} \le {H_i}$ for any $i \in \mathbb{N}$, the assertion follows from the definition of ${L_i}$. \par
(iii) Since $N/H \cong {(N/H)^g} = N/{H^g}$, we see that $\pi (N/H) = \pi (N/{H^g})$. Therefore, as ${H^g}H/{H^g} \cong H/({H^g} \cap H)$ and ${H^g}H/{H^g} \le N/{H^g}$, we can conclude that $\pi (H/({H^g} \cap H)) \subseteq \pi (N/H)$ for any $g \in G$. Then it follows from Remak's theorem that $\pi (H/\bigcap\nolimits_{g \in X} {({H^g} \cap H)} ) \subseteq \pi (N/H)$ for any finite subset $X \subseteq G$. Therefore, by the definition of ${L_i}$, we have $\pi (H/{L_i}) \subseteq \pi (N/H)$ and hence, as ${L_i} \le {H_i}$,  we can conclude that $\pi (H/{H_i}) \subseteq \pi (N/H)$ for any $i \in \mathbb{N}$. 
\end{proof}

	\begin{lemma}\label{Lemma 2.1.4} Let $H$ be a finitely generated nilpotent group and $\{L_i \, |\, i \in \mathbb{N}\}$  be normal subgroups of finite index in $H$ which contain the derived subgroup $H'$ of $H$. Let $p$ be a prime integer and let $\{S_i \, |\, i \in \mathbb{N}\}$  be subgroups of $H$ such that ${L_i} \le {S_i}$ and ${S_i}/{L_i}$ is the Sylow $p$-subgroup of $ {H}/{L_i} $ for each $ i \in \mathbb{N}\ $. Suppose that there is a descending chain  of normal subgroups ${H_i} \le H$ of finite index in $H$ such that the quotient group $H/(\bigcap_{i\in \mathbb{N}} H_i)$ is free abelian, $\bigcap\nolimits_{j = 1}^i {{L_j}}  \subseteq {H_i}$ for each $ i \in \mathbb{N}\ $  and there is $n \in \mathbb{N}$ such that $p \notin \pi ({H_n}/{H_i})$ for all $i \ge n$. If $D$ is a subgroup of $H$ such that $D \le {S_i}$ for each $i \in \mathbb{N}$ then $D\leq \bigcap_{i\in \mathbb{N}} H_i $.
\end{lemma}

\begin{proof} Put $K = \bigcap_{i\in \mathbb{N}} H_i $ and suppose that there is an element $d \in D\backslash K$. Since the quotient group  $H/K$ is torsion-free and $\left| {H/{H_n}} \right| < \infty $, we can assume that $d \in {H_n}\backslash K$. Then, as $K = \bigcap_{i\in \mathbb{N}} H_i $, there is $i > n$ such that $d \in {H_n}\backslash {H_i}$ and hence, as $\bigcap\nolimits_{j = 1}^i {{L_j}}  \subseteq {H_i}$, we have $d \in {H_n}\backslash (\bigcap\nolimits_{j = 1}^i {{L_j}} )$. Since $D \le {S_i}$ for each $ i\in \mathbb{N} $, we see that $d \in \bigcap\nolimits_{j = 1}^i {{S_j}} $ and, as each quotient group ${S_i}/{L_i}$ is a $p$-group, it is not difficult to show that ${d^{{p^m}}} \in \bigcap\nolimits_{j = 1}^i {{L_j}} $ for some $ m \in \mathbb{N} $. Then, as $\bigcap\nolimits_{j = 1}^i {{L_j}}  \le {H_i}$, we have ${d^{{p^m}}} \in {H_i}$ but it is impossible, because $d \in {H_n}\backslash {H_i}$ and $p \notin \pi ({H_n}/{H_i})$. The obtained contradiction shoes that 
$D\leq K = \bigcap_{i\in \mathbb{N}} H_i $. 
\end{proof}

\subsection{Ring-theoretic lemmas}  \label{Subsection 2.2}

\begin{lemma}\label{Lemma 2.2.1} Let $G$ be an abelian group and let $H$ be a dense subgroup of $G$. Let $k$ be a field and $P$ be a prime ideal of $kG$. Then: \par
         (i) the group ring $kG$ is integer over $kH$;\par
         (ii) the ideal $P$ is maximal if and only if $P \cap kH$ is a maximal ideal of  $kH$; \par
         (iii) if ${P_1}$ is an ideal of $kG$ such that $P \le {P_1}$ and $P \cap kH = {P_1} \cap kH$  then  $P = {P_1}$     \end{lemma}

          \begin{proof} (i) Since $H$ is a dense subgroup of $G$, for any element $g \in A$ there is a positive integer  $ n\in \mathbb{N} $ such that ${g^n} = a \in kH$ and hence $g$ is a root of the polynomial ${X^n} - a$. Thus, each element of $G$ is integer over $kH$. It easily implies that the group ring $kG$ is integer over $kH$. \par
        (ii) The assertion follows from (i) and \cite[Chap. V, \S 2, Proposition 1]{Bour}. \par
        (iii) The assertion follows from (ii) and \cite[Chap. V, \S 2, Corollary 1]{Bour}.
\end{proof}

	\begin{lemma}\label{Lemma 2.2.2} Let $k$ be a field and $G$ be a group. Let $I$ be a right proper ideal of the group ring $kG$ such that $I = (I \cap kH)kG$, where $H$ is a subgroup of $G$. If there is a subgroup $X$ of $G$ such that $X \cap H = 1$ then $kX \cap I = 0$. If in addition the ideal is $I$ two-sided then the quotient ring $kG/I$ contains a subring isomorphic to $kX$.
\end{lemma}

\begin{proof} Suppose that there is a nonzero element $a \in kX \cap I$. Since $I = (I \cap kH)kG$, the element $a$ can be written in the form $a = \sum\nolimits_{i = 1}^n {{a_i}} {t_i}$, where all ${a_i} \in I \cap kH$, all ${t_i} \in T$ and $T$ is a right transversal for $H$ in $G$. Since the ideal $I$ is proper, we see that $|Supp({a_1})| \ge 2$ and hence the element ${a_1}$ can be presented in the form ${a_1} = \sum\nolimits_{i = 1}^m {{\alpha _i}} {h_i}$, where all ${\alpha _i} \in k$, all ${h_i} \in H$ and $m \ge 2$. Then, as $a \in kX \cap I$, we see that ${h_1}{t_1},{h_2}{t_1} \in X$ and hence ${h_1}{t_1}{({h_2}{t_1})^{ - 1}} = {h_1}{h_2}^{ - 1} \in X$ but $1 \ne {h_1}{h_2}^{ - 1} \in H$ and it contradicts $X \cap H = 1$. Thus, $kX \cap I = 0$. \par
Suppose that the ideal $I$ is two-sided then then $(kX + I)/I \cong kX/(kX \cap I) = kX$ 
\end{proof}

Let $A$ be a normal subgroup of a group ${\rm{ G}}$. If  $R$ is a ring then the action by conjugations of ${\rm{ G}}$ on $A$ can be extended to the action of $H$ on the group ring $RA$. We say that an ideal $I$ of $RA$ is $H$-invariant if ${I^g} = I$ for any element $g \in H$.  

	\begin{lemma}\label{Lemma 2.2.3} Let $D$ be a group, ${D_1}$ be a subgroup of $D$, 
	 $A$ be an abelian normal subgroup of D and $ A_1= A \cap D_1 $. Let $R$ be a ring, ${R_1}$ be a subring of $R$ and suppose that the group ring $RA$ has a proper $D$-invariant ideal $P$. Let ${P_1} = {R_1}{A_1} \cap P$ then ${P_1}{R_1}{D_1} = {R_1}{D_1} \cap PRD$.
\end{lemma}

	\begin{proof}  Evidently, ${P_1}{R_1}{D_1} \subseteq {R_1}{D_1} \cap PRD$. Let $a \in {R_1}{D_1} \cap PRD$ then the element $a$ can be presented in the form $a = \sum\nolimits_{i = 1}^n {{a_i}} {t_i}$, where ${a_i} \in P$, ${t_i} \in T$ and $T$ is a right transversal for $A$ in $D$. In its tern each element ${a_i}$ can be presented in the form ${a_i} = \sum\nolimits_{j = 1}^{{k_i}} {{\alpha _{ij}}} {s_{ij}}$, where ${\alpha _{ij}} \in R$, ${s_{ij}} \in A$ and, as the ideal $P$is proper, $2 \le {k_i}$ for all $i,j$. Since $a \in {R_1}{D_1}$, we see that ${\alpha _{ij}} \in {R_1}$ and ${s_{ij}}{t_i} \in {D_1}$ for all $i,j$. For each $i$ we have 
$({s_{ij}}{t_i}){({s_{i1}}{t_i})^{ - 1}} = {s_{ij}}{({s_{i1}})^{ - 1}} = {r_{ij}} \in {D_1} \cap A = {A_1}$, where $j \ne 1$. Put ${r_{i1}} = e$ then for each ${s_{ij}}$ we have ${s_{ij}} = {r_{ij}}{s_{i1}}$, where  ${r_{ij}} \in {A_1}$. Therefore, each ${a_i} = \sum\nolimits_{j = 1}^{{k_i}} {{\alpha _{ij}}} {r_{ij}}{s_{i1}} = (\sum\nolimits_{j = 1}^{{k_i}} {{\alpha _{ij}}} {r_{ij}}){s_{i1}}$ and, as ${a_i} \in P$, we see that ${b_i} = \sum\nolimits_{j = 1}^{{k_i}} {{\alpha _{ij}}} {r_{ij}} = {a_i}{({s_{i1}})^{ - 1}} \in P \cap {R_1}{A_1} = {P_1}$. Since ${a_i} = {b_i}{s_{i1}}$, we have $a = \sum\nolimits_{i = 1}^n {(\sum\nolimits_{j = 1}^{{k_i}} {{\alpha _{ij}}} {r_{ij}}({s_{i1}}{t_i}))} $, where ${r_{ij}}({s_{i1}}{t_i}) \in {D_1}$, and hence, as ${r_{ij}} \in {A_1}$, we have $({s_{i1}}{t_i}) \in {D_1}$. Then $a = \sum\nolimits_{i = 1}^n {{b_i}} ({s_{i1}}{t_i})$, where ${b_i} \in {P_1}$ and ${s_{i1}}{t_i} \in {D_1}$, and we can conclude that $a \in {P_1}{R_1}{D_1}$. Therefore, ${R_1}{D_1} \cap PRD \subseteq {P_1}{R_1}{D_1}$ and, as${P_1}{R_1}{D_1} \subseteq {R_1}{D_1} \cap PRD$, we see that    ${P_1}{R_1}{D_1} = {R_1}{D_1} \cap PRD$. 
\end{proof}

	Let $R$ be a ring and $M$ be an $R$-module. The module $M$ is said to be R-torsion if $an{n_R}(a) \ne 0$ for any element $a \in M$. If $an{n_R}(a) = 0$ for any nonzero element $a \in M$ then the module $M$ is said to be $R$-torsion-free. The module $M$ is said to be uniform if $U \cap V \ne 0$ for any nonzero submodules $U$ and $V$ of $M$. The ring $R$ is said to be uniform if the right module ${R_R}$ is uniform.

	\begin{lemma}\label{Lemma 2.2.4} Let $S$ be a skew group ring of a uniform domain $R \le S$ and an infinite cyclic group $\left\langle {{g^n}} \right\rangle $.  Let $J$ be a nonzero right ideal of $S$ then the quotient module $M = S/J$ has a finitely generated $R$-submodule $U$ such that the quotient module $M/U$ is $R$-torsion.
\end{lemma}

	\begin{proof} By the definition of skew group rings, $ S=\oplus_{i\in \mathbb{Z}}Rg^i $ and ${R^{{g^i}}} = {g^{ - i}}R{g^i} = R$ for all integers $i$. Let $ T=\oplus_{i\in \mathbb{N}\cap \{0\}}Rg^i $, ${T_n} =  \oplus _{i = 0}^nR{g^i}$ and ${T_{ - 1}} = 0$ then ${T_n} = {T_{n - 1}} \oplus R{g^n}$, Put $\tilde T = T + J/J$ and ${\tilde T_n} = {T_n} + J/J$ then ${\tilde T_{n - 1}} \le {\tilde T_n}$ and $ \tilde T=\bigcup_{i\in \mathbb{N}\bigcup \{0\}}\tilde T_n $. Since $J \ne 0$, there is an integer $k$ which is minimal with respect to ${T_k} \cap J \ne 0$. Therefore, as $J$ is a right ideal of $S$, for any $n \ge k$ we have ${T_n} \cap J > {T_{n - 1}} \cap J$ and it easily implies that ${\tilde T_n}/{\tilde T_{n - 1}} \cong R/{J_n}$, where ${J_n}$ is a nonzero right ideal of $R$. Then, as the domain $R$ is uniform, each factor ${\tilde T_n}/{\tilde T_{n - 1}} \cong R/{J_n}$ is $R$-torsion for any $n \ge k - 1$. Put $U = {\tilde T_{k - 1}}$ then, as  $ \tilde T=\bigcup_{i\in \mathbb{N}\bigcup \{0\}}\tilde T_n $ and the domain $R$ has no zero divisors, we can conclude that the quotient module $\tilde T/U$ is $R$-torsion. \par
	Let ${M_n} = T \oplus ( \oplus _{i = 1}^nR{g^{ - i}})$, where , ${M_0} = T$ and ${\tilde M_n} = ({M_n} + J)/J$ then ${\tilde M_{n - 1}} \le {\tilde M_n}$, and $ M=\bigcup_{i\in \mathbb{N}\bigcup \{0\}}\tilde M_n $. So, the above arguments show that the quotient module $M/\tilde T$ is $R$-torsion. Then, as the domain $R$ has no zero divisors and the quotient module $\tilde T/U$ is $R$-torsion, we can conclude that the quotient module $M/U$ is $R$-torsion. 
\end{proof}

\section{Prime and maximal ideals in group rings of abelian groups}  \label{Section 3}
\subsection{Subfields generated by multiplicative subgroups of finite total rank} \label{Subsection 3.1}

	\begin{proposition}\label{Proposition 3.1.1} Let $F$ be an algebraic closed field and let $k$ be a finitely generated subfield of $F$. Let $K$ be a subfield of $F$ generated by $k$ and all roots of unity. Then the multiplicative group $K^*$ of the field $K$ may be presented in the form $K^* = T \times A$, where $T$ is a divisible torsion abelian group and $A$ is a free abelian group. 
\end{proposition}

	\begin{proof} At first, we show that the Galois group $G = Gal(K/k)$ of the field extension $K/k$ is abelian. Let $T$ be the torsion subgroup of the multiplicative group $K^*$ of the field $K$. Then $T$ coincides with the set of all roots of unity and $K = k(T)$. Evidently, each automorphism $g \in G$ induces an automorphism ${g_T} \in Aut(T)$ of the group $T$ and the mapping $\varphi :g \mapsto {g_T}$ is a homomorphism of the group $G = Gal(K/k)$ into  $Aut(T)$. It is easy to note that $Ker\,\varphi  = {C_G}(T)$. Therefore, if $g \in Ker\,\varphi $ then $g$ centralizes all roots of unity from $T$ and hence, as $K = k(T)$ and $g$ centralizes the elements of the subfield $k$, we can conclude that $Ker\,\varphi  = 1$. Thus, $\varphi $ is a monomorphism of the group $G$ into $Aut(T)$. If follows from \cite[ Theorem 127.3]{Fuch1973} that $T = { \times _{p \ne char\,k}}{C_{{p^\infty }}}$, where $p$ runs all prime integers such that $p \ne char\,k$ and ${C_{{p^\infty }}}$ is a quasi-cyclic $p$-group. As ${C_{{p^\infty }}}$ is the $p$-component of the group $T = { \times _{p \ne char\,k}}{C_{{p^\infty }}}$, it follows from \cite[\S 113, Assertion g)]{Fuch1973} that $Aut(T) = { \times _{p \ne char\,k}}Aut({C_{{p^\infty }}})$. It is well known that $Aut({C_{{p^\infty }}})$ is isomorphic to the group of units of the ring ${J_p}$ of $p$-adic numbers (see \cite[\S 113, Example 3]{Fuch1973}). Thus, we can conclude that the group $Aut(T) = { \times _{p \ne char\,k}}Aut({C_{{p^\infty }}})$ is abelian and, as $\varphi $ is a monomorphism of the group $G$ into $Aut(T)$, we see that the group $G$ is abelian. Then it follows from \cite[Corollary of Theorem 2]{May1972} that $K^* = T \times A$. 
\end{proof}

Let $F$ be a field, $K$ be a subfield of $F$ and $X$ be a subset of $F$. Then $F(X)$ denotes 
a subfield of $F$ generated by $F$ and $X$ and $F[X]$ denotes a subring of $F$ generated by $F$ and $X$. The next proposition can be considered as a generalization of \cite[Chap. VIII, Theorem 13]{Lang1965} to the case of infinitely dimensional extensions.

	\begin{proposition}\label{Proposition 3.1.2}  Let $K$ be a subfield of a field $F$ and suppose that the field $K$ contains all roots of unity. Let $G$ be a subgroup of the multiplicative group $F^*$ of the field $F$ and let $H = K^* \cap G$. Suppose that the quotient group $G/H$ is torsion and $char\,K \notin \pi (G/H)$. Then $K(G) = K{ \otimes _{KH}}KG = { \oplus _{t \in T}}Kt$, where $T$ is a transversal to $H$ in  $G$. 
\end{proposition}

	\begin{proof} See \cite[Proposition 2.1.1]{Tush2006}. 
\end{proof}

If $G$ is a group then $Sp(G)$ denotes the set of prime integers $p$ such that the group $G$ has an infinite $p$-section. 

\begin{proposition}\label{Proposition 3.1.3} Let $F$ be an algebraically closed field and let $k$ be a finitely generated subfield of $F$. Let $G$ be a subgroup of finite total rank of the multiplicative group $F^*$ of  $F$ such that $char\,k \notin Sp(G)$. Then there is a subfield $K$ of $F$ such that $k \le K$, $N = K^* \cap G$ is a minimax subgroup of $G$ and $K(G) = K{ \otimes _{kN}}kG = { \oplus _{t \in T}}Kt$, where $T$ is a transversal for  $N$ in  $G$. 
\end{proposition}

\begin{proof} Since $char\,k \notin Sp(G)$, we can choose a finitely generated dense subgroup $H$ of $G$ such that $char\,k \notin \pi (G/H)$. It is not difficult to note that the field $k(H)$ is finitely generated. Let $K$ be a subfield of $F$ generated by $k(H)$ and all roots of unity. Then, by Proposition 3.1.1, the multiplicative group $K^*$ of the field $K$ may be presented in the form $K^* = T \times A$, where $T$ is divisible torsion abelian group and $A$ is a free abelian group. Let ${T_1}$ be the torsion subgroup of the group $G$, as the group $G$ has finite total rank, the subgroup ${T_1}$ is Chernikov. Let $N = G \cap K^*$ then the quotient group $NT/T$ is a subgroup of the free abelian quotient group $K^*/T$ and hence the quotient group $NT/T$ is free abelian (see \cite[Theorem 14.5]{Fuch1973}). Therefore, the quotient group $N/(N \cap T) \cong NT/T$ if free abelian and, as the subgroup $N = G \cap K^*$ has finite total rank, we see that the quotient group $D/{T_2}$ is finitely generated free abelian, where ${T_2} = N \cap T$. Evidently ${T_2} \le {T_1}$ and hence the subgroup ${T_2}$ is Chernikov. Then, as the quotient group $N/{T_2}$ is finitely generated free abelian, we can conclude that the subgroup $N = G \cap K^*$ is minimax. By Proposition 3.1.2, $K(G) = K{ \otimes _{kN}}kG = { \oplus _{t \in T}}Kt$, where $T$ is a transversal to the subgroup $N$ in the group $G$. 
\end{proof}

\begin{theorem}\label{Theorem 3.1.4} Let $F$ be a field and let $k$ be a finitely generated subfield of $F$. Let $G$ be a subgroup of finite total rank of the multiplicative group $F^*$ of  $F$   such that $char\,k \notin Sp(G)$. Then there exist a finitely generated subgroup $H$ of $G$ such that $k(G) = k(H){ \otimes _{kH}}kG = { \oplus _{t \in T}}k(H)t$ and $k[G] = k[H]{ \otimes _{kH}}kG = { \oplus _{t \in T}}k[H]t$, where $T$ is a transversal for   $H$ in   $G$. 
\end{theorem}

\begin{proof} Evidently, there is no harm in assuming that the field $F$ is algebraically closed.  Then, by Proposition 3.1.3, there is a subfield $K$ of $F$ such that $k \le K$, $N = K^* \cap G$ is a minimax subgroup of $G$ and $K(G) = K{ \otimes _{kN}}kG = { \oplus _{t \in T_N}}Kt$, where $T_N$ is a transversal for  $N$ in $G$. Since $k(N) \le K$, it easily implies that $k(G) = k(N){ \otimes _{kN}}kG = { \oplus _{t \in T_N}}k(N)t$. \par 

Since $k$ is a finitely generated field, there is a finitely generated commutative domain $d \le k$ such that $k$ is the field of fractions of $d$. Then, by \cite[Theorem 1.1]{Sega2001}, there exists a finitely generated dense subgroup $H$ of $N$ such that $d[N] = d[H]{ \otimes _{dH}}dN = { \oplus _{t \in T_H}}d[N]t$, where $T_H$ is a transversal  for $H$ in  $N$. Evidently, $k(N)$ is the field of fractions of $d[N]$. Since the subgroup $H$ is dense in $N$, we see that all elements of $N$ are algebraic over $k(H)$ and hence $k(N) = \sum\nolimits_{t \in T_H} {k(H)t} $ because $k(N)$ is an algebraic extension of $k(H)$ and $\sum\nolimits_{t \in T_H} {k(H)t} $ is a subring of $k(N)$. Thus, it is sufficient to show that the sum $k(N) = \sum\nolimits_{t \in T_H} {k(H)t} $ is direct. Suppose that the sum $k(N) = \sum\nolimits_{t \in T_H} {k(H)t} $ is not direct. Then there are elements ${t_i} \in T_H$ and $0 \ne {a_i} \in k(H)$, where $1 \le i \le n$, such that $\sum\nolimits_{i = 1}^n {{a_i}{t_i}}  = 0$. As $k(H)$ is the field of fractions of $d[H]$, there is an element $b \in d[H]$ such that $0 \ne {a_i}b \in d[H]$ for all $1 \le i \le n$. Therefore, $\sum\nolimits_{i = 1}^n {{a_i}b{t_i}}  = 0$, where $0 \ne {a_i}b \in d[H]$ for all $1 \le i \le n$, but this contradicts $d[N] = { \oplus _{t \in T_H}}d[H]t$ and hence $k(N) = \oplus_{t \in T_H} {k(H)t} = k(H){ \otimes_{kH}}kN$. Then, as $k(G) = k(N){ \otimes _{kN}}kG$, it follows from \cite[Chap. 2, Lemma 2.1]{Karp} that $k(G) = k(H){ \otimes _{kH}}kG = { \oplus _{t \in T}}k(H)t$, where $T$ is a transversal for $H$ in $G$. 
\par
Evidently, $k[G] = \sum\nolimits_{_{t \in T}} {k[H]t} $ then, as $k[H] \le k(H)$ and the sum $k(G) = { \oplus _{t \in T}}k(H)t$ is direct, we can conclude that $k[G] = k[H]{ \otimes _{kH}}kG = { \oplus _{t \in T}}k[H]t$.

\end{proof}

\begin{corollary}\label{Corollary 3.1.5} Let $k$ be a finitely generated field and let $G$ be an abelian group of finite total rank such that $char\,k \notin Sp(G)$. Then any faithful irreducible representation $\varphi $ of the group $G$ over the field $k$ is induced from a representation of  some finitely generated subgroup of $G$.
\end{corollary}

\begin{proof} Let $F$ be a module of the representation $\varphi $. Since the module $F$ is irreducible, we see that $F = akG \cong kG/I$, where $0 \ne a \in F$ and $I = An{n_{kG}}(a)$ is a maximal ideal of $kG$, and hence the quotient ring $F = kG/I$ is a field. Therefore, $k$ is a subfield of $F$ and $F = k[G]$. Then, by Theorem 3.1.4, there exist a finitely generated subgroup $H$ of $G$ such that $k[G] = k[H]{ \otimes _{kH}}kG = { \oplus _{t \in T}}k[H]t$, where $T$ is a transversal to the subgroup $H$ in the group $G$, and the assertion follows. 
\end{proof}

\subsection{Faithful prime ideals of group rings of abelian groups of finite total rank} \label{Subsection 3.2}

		\begin{theorem}\label{Theorem 3.2.1} Let $k$ be a finitely generated field and let $G$ be an abelian group of finite total rank such that $char\,k \notin Sp(G)$. Let $P$ be a prime faithful ideal of the group ring $kG$. Then:\par
(i)	  there exists a finitely generated subgroup $H$ of $G$ such that $P = (P \cap kH)kG$; \par
(ii)	the ideal $P$ is finitely generated;\par
	\end{theorem}

\begin{proof} (i) Let $R = kG/P$ and let $F$ be the field of fractions of the commutative domain $R$. As the ideal $P$ is faithful, we can assume that $G \le F^*$. By Theorem 3.1.4, there exists a subgroup $H \le G$ such that $k[G]  =  k[H]{ \otimes _{KH}}kG = { \oplus _{t \in T}}k[H]t$, where $T$ is a transversal to $H$ in $G$.  Evidently, $ann_{kG}k[G] = P$ and $ann_{kH}k[H] = P \cap kH$. Then it follows from \cite[Chap. 2, Theorem 3.4(i)]{Karp} that $an{n_{kG}}k[G] = Id((an{n_{kH}}k[H])kG)$, where $Id((an{n_{kH}}k[H])kG)$ is the sum of all two sided ideals of $k[G]$ contained in $(an{n_{kH}}k[H])kG$. Since the ring $kG$ is commutative, we see that $Id((an{n_{kH}}k[H])kG) = (an{n_{kH}}k[H])kG$ and hence $P = (P \cap kH)kG$. \par
        (ii) Since the subgroup $H$ is finitely generated, the group ring $kH$ is Noetherian. Hence $P \cap kH$ is a finitely generated ideal of $kH$, and the assertion follows.\par
         \end{proof}

\begin{theorem}\label{Theorem 3.2.2} Let $k$ be a finitely generated field and let $G$ be a countable abelian group such that $char\,k \notin Sp(G)$. Suppose that the group ring $kG$ has a faithful prime ideal then:\par 
(i)  each faithful prime ideal of $kG$ is finitely generated if and only if the group $G$ has finite total rank;\par
(ii) if the group ring $kG$ has a faithful maximal finitely generated ideal then the group $G$ has finite total rank. 
\end{theorem}

\begin{proof} (i) If the group $G$ has finite total rank, then it follows from Theorem 3.2.1(ii) that each faithful prime ideal of the group ring $kG$ is finitely generated. \par
	Suppose that the group ring $kG$ has a faithful prime ideal and each such ideal is finitely generated. Let $P$ be a finitely generated faithful prime ideal of $kG$ and let $F$ be the field of fractions of the commutative domain $R = kG/P$. Since the ideal $P$ is faithful, we can assume that that $G \le F^*$. Then it follows from \cite[Theorem 127.3]{Fuch1973} that the torsion subgroup $T$ of the group $G$ is locally cyclic and $char \, k \notin \pi (T)$. \par
Suppose that the set $\pi (T)$ is infinite. Since the ideal $P$ is finitely generated, there is a finitely generated subgroup $H$ of $G$ such that $P = (kH \cap P)kG$. As the set $\pi (T)$ is infinite, it follows from Lemma 2.1.1(i) that there is a nontrivial subgroup $X$ of $T$ such that $X \cap H = 1$. Then it follows from Lemma 2.2.2 that the quotient ring $kG/P$ has a subring isomorphic to $kX$. However, it is well known that the group ring $kX$ of the torsion group $X$ has zero divisors but it is impossible because the ideal $P$ is prime. So, a contradiction is obtained and hence the set $\pi (T)$ is finite. Since the subgroup $T$ is locally cyclic, it implies that $T$ is Chernikov and hence the total rank ${r_t}(T)$ is finite.\par
Suppose that the total rank ${r_t}(G)$ is infinite then, as ${r_t}(T)$ is finite, the rank $r(G/T)$ of the torsion-free quotient group $G/T$ is infinite. Now, we will show that in this case the group ring has a faithful prime ideal which does not admit a finite set of generators. Let $F$ be the algebraic closure of the field $k$ if the field $k$ is infinite and $F$ be the algebraic closure of a simple transcendent extension $k(t)$ of $k$. It is not difficult to note that the field $F$ has finite transcendence degree over $k$ if the field $k$ is finite. Then it follows from \cite[Theorem 127.3]{Fuch1973} that  $ F^*= (\times_{p_i \in \sigma}{C_{p_i^\infty }} )\times(\times_{i\in \mathbb{N}}\mathbb{Q}_{i} )$ if  $char \, k = 0$  and $ F^*= (\times_{p_i \in \sigma_p}C^\infty_{p_i} )\times(\times_{i\in \mathbb{N}}\mathbb{Q}_{i} ) $
  if $char\, k = p > 0$  , where $  \sigma = \{p_i \, | \, i \in \mathbb{N} \} $ is the set of all prime integers, $\sigma_p = \sigma \setminus \{ p \}$, ${C_{p_i^\infty }}$ is a quasicyclic ${p_i}$-group and $\mathbb{Q}_{i}$ is a group isomorphic to the additive group of rational numbers. \par
 
Let $\bar G$ be the divisible hull of the group $G$. Then, as $T = t(G)$ is a locally cyclic $\pi $-group, where $\pi  = \pi (T)$ is the finite set of prime divisors of orders of elements of $T$, it follows from \cite[Theorem 23.1]{Fuch1973} that 
$\bar G = (\times_{p_i \in \pi (G)}{C_{p_i^\infty }})\times (\times_{i \in \mathbb{N}}  \mathbb{Q}_{i}) )$. 
As it was mentioned above, if $char \, k = p > 0$ then $p \notin \pi $. It easily implies that there is a monomorphism of $\bar G$ into $F^*$ and hence, as $G \le \bar G$, there is a monomorphism $\varphi :G \to F^*$. Since $k \le F$, the monomorphism $\varphi $ can be extended to a ring homomorphism $\varphi :kG \to F$. As $kG/Ker\varphi  \cong \varphi (kG) \le F$, we can conclude that $P = Ker\varphi $ is a faithful prime ideal of $kG$. \par 
Suppose that the ideal $P$ is finitely generated then there is a finitely generated subgroup $H$ of $P$ such that $P = (P \cap kH)kG$. Since the torsion-free quotient group $G/T$ has infinite rank, it follows from Lemma 2.1.1(ii) that there is a free abelian subgroup $Y \le G$ of infinite rank such that $Y \cap H = 1$. Then it follows from Lemma 2.2.2 that $kY \cap P = 0$ and, as $kG/P \cong \varphi (kG) \le F$, we can conclude that the transcendence degree of the field $F$ over $k$ is infinite that leads to a contradiction. \par
(ii) Suppose that $kG$ has a faithful maximal finitely generated ideal $P$ then, as ideal $P$ is prime, the arguments of (i) show that ${r_t}(T) < \infty $. Therefore, if  ${r_t}(G) = \infty $ then ${r_t}(G/T) = \infty $. Since the ideal $P$ is finitely generated, there is a finitely generated subgroup $H \le G$ such that $(kH \cap P)kG = P$. Let $L = i{s_G}(H)$ then the quotient group $G/L$ is torsion-free of infinite rank and $(kL \cap P)kG = P$. Evidently, we can consider the quotient ring $kG/P$ as an abelian crossed product $\tilde k* \tilde G$, where 
$\tilde k = kL/(kL \cap P)$ is a domain. and $\tilde G = G/L$ is a torsion-free abelian group of infinite rank. 
It is not difficult to note that $a\tilde k* \tilde G$ is a proper ideal of $\tilde k* \tilde G$ for any $0 \ne a \in \tilde k* \tilde G$ but it is impossible because the ideal $P$ is maximal. The obtained contradiction shows that ${r_t}(G) < \infty $.
\end{proof}

	Let $R$ be a ring and $G$ be a group. An $RG$-module $W$ is said to be fully faithful if for any nonzero submodule $V$ of $W$ the centralizer ${C_G}(V)$ of $V$ in $G$ is trivial (i.e. ${C_G}(V) = 1$). It is easy to note that ${C_G}(V) = {(An{n_{RG}}(V))^\dag }$ and hence $W$ is fully faithful if and only if ${(An{n_{RG}}(V))^\dag } = 1$ for any nonzero submodule $V$ of $W$. Therefore, if the group $G$ is abelian then the module $W$ is fully faithful if and only if ${(An{n_{RG}}(a))^\dag } = 1$ for any element $0 \ne a \in W$. 

	\begin{theorem}\label{Theorem 3.2.3} Let $k$ be a finitely generated field and let $G$ be an abelian group of finite total rank such that $char\,k \notin Sp(G)$. Let $W$ be a fully faithful $kG$-module then there is an element $0 \ne a \in W$ such that $An{n_{RG}}(a)$ is a prime faithful ideal of $kG$. 
\end{theorem}

	\begin{proof} Since the group $G$ has finite total rank, there is an ascending chain 
$ \{ H_i \, | \, i \in \mathbb{N}\} $	 of finitely generated dense subgroups of $G$ such that $\left| {{H_{i + 1}}:{H_i}} \right| < \infty $ for each $ i \in \mathbb{N} $ and $ \bigcup_{i \in \mathbb{N} } H_i = G $. By \cite[Proposition 1.6]{Pass1989}, each group ring $k{H_i}$ is Noetherian because finitely generated abelian groups are polycyclic. Then it follows from \cite[Chap. IV, \S 2, Corollary 1]{Bour} that there is a prime ideal ${P_1} \le k{H_1}$ which is maximal with respect to ${P_1} = ann_{k{H_1}}({a_1})$ for some nonzero element $0 \ne {a_1} \in W$. Put ${W_1} = {a_1}kG$ then $ann_{k{H_1}}({W_1}) = {P_1}$. Since the module $W$ is fully faithful, ${(ann_{kG}({a_1}))^\dag } = 1$ and it easily implies that the ideal ${P_1}$ is faithful. The above arguments show that there is a prime ideal ${P_2} \le k{H_2}$ which is maximal with respect ${P_2} = ann_{k{H_2}}({a_2})$ for some nonzero element $0 \ne {a_2} \in {W_1}$. As $ann_{k{H_1}}({W_1}) = {P_1}$, we see that ${P_1} \le {P_2}$. Put ${W_2} = {a_2}kG$ then $ann_{k{H_2}}({W_2}) = {P_2}$. So, continuing this process we obtain an ascending chain $ \{ P_i \, | \, i \in \mathbb{N}\} $	, where each ${P_i}$ is a faithful prime ideal of $k{H_i}$. Therefore,  $ \bigcup_{i \in \mathbb{N} } P_i = P $ is a prime faithful ideal of $kG$. Since each ideal ${P_i}$ is maximal with respect $an{n_{{W_{i - 1}}}}({P_i}) \ne 0$, it is easy to note that ${P_i} = k{H_i} \cap {P_j}$ for any $j \ge i$ and hence ${P_i} = k{H_i} \cap P$. Then it follows from Theorem 3.2.1(ii) that there is a positive integer $ i \in \mathbb{N} $ such that $P = (k{H_i} \cap P)kG$ and hence $P = {P_i}kG$. Therefore, $P = {P_i}kG$ annihilates $a = {a_i} \ne 0$ and hence $ann_{kG}(a) \ge P$. Since $ann_{kG}(a) \cap k{H_i} = ann_{k{H_i}}(a) = {P_i} = P \cap k{H_i}$, it follows from Lemma 2.2.1(iii) that $an{n_{kG}}(a) = P$. 
\end{proof}

\section{On certain induced modules over group rings of nilpotent groups}  \label{Section 4}

\subsection{Rings of quotients and induced modules over crossed products of nilpotent groups}  \label{Subsection 4.1}  

A ring $R * G$ is called a crossed product of a ring $R$ and a group $G$ if $R \le R * G$ and there is an injective mapping ${\varphi ^*}:g \mapsto \bar g$ of the group $G$ to the group of units $U(R * G)$ of  $R * G$ such that each element $a \in R * G$ can be uniquely presented as a finite sum $a = \sum\nolimits_{g \in G} {{a_g}} \bar g$, where ${a_g} \in R$. The addition of two such sums is defined component-wise. The multiplication is defined by the formulas $\bar g\bar h = t(g,h)\overline {gh} $ and $r\bar g = \bar g({\bar g^{ - 1}}r\bar g)$, where $g,h \in G$,$r \in R$, ${\bar g^{ - 1}}r\bar g \in R$ and $t(g,h)$ is a unit of  $R$ (see \cite{Pass1989}). \par
If $a = \sum\nolimits_{g \in G} {{a_g}} \bar g \in R * G$, then the set $Supp(a)$ of  elements $g \in G$ such that ${a_g} \ne 0$ is called the support of the element $a$. Let $H$ be a subgroup of  $G$ then the set of  elements $a \in R * G$ such that $Supp(a) \subseteq H$ forms a crossed product $R * H$ contained in $R * G$. If the  subgroup $H$ is normal and $P$ is a $\bar G$-invariant ideal of the ring $R * H$ then it is not difficult to verify that  the quotient ring ${R * G} / {PR * G}$ is a crossed product $({R * H}/ {P}) * (G /H)$ of the quotient ring ${{R * H} /P}$ and the quotient group $G / H$. In particular, if $RG$ is a group ring, $H$ is a normal subgroup of the group  $G$ and $P$ is a $G$-invariant ideal of the group ring $RH$ then the quotient ring ${{RG} / {PRG}}$ is a crossed  product $({RH} /P) * (G / H)$ of the quotient ring ${RH} / P$ and the quotient group $G / H$. The last example  shows the main way in which crossed products arise in studying of group rings and modules over them. \par
If $V$ is a $R * H$-module then we can define the tensor product $V{ \otimes _{R * H}}R * G$ which is a right $R * G$-module. The arguments of \cite[Chap. 2, Corollary 1.2(i)]{Karp} show that an $R * G$-module $M = V{ \otimes_{R * H}}R * G$ if and only if $M = { \oplus_{t \in T}}V\bar t$, where $T$ is a right transversal for $H$ in  $G$. 
	
\begin{proposition}\label{Proposition 4.1.1} Let $G$ be a group, $D$ be a normal subgroup of $G$ and $R$ be 
a ring. Suppose that there exists a crossed product $R * G$ such that $R * D$ is an Ore domain. Then there exists a partial right ring of quotients $R * G{(R*D)^{ - 1}} = \left\{ {r \cdot {s^{ - 1}}|r \in R * G,s \in R*D} \right\}$ and: \par
(i) $R * G{(R*D)^{ - 1}}$ is a crossed product $(R * D{(R*D)^{ - 1}}) * (G/D) = $  
${ \oplus _{t \in T}}(R * D{(R*D)^{ - 1}})\bar t = (R * D{(R*D)^{ - 1}}){ \otimes _{R * D}}R * G$, where $R * D{(R*D)^{ - 1}}$ is a division ring and $T$ is a right transversal for $D$ in $G$;\par
(ii) if the quotient group $G/D$ is polycyclic-by-finite then the ring $R * G{(R*D)^{ - 1}}$ is Noetherian;\par
	(iii) if $M = aR * G$ is a cyclic $R * G$-module which is $R * D$-torsion-free then there exists a cyclic $R * G{(R*D)^{ - 1}}$-module $W = aR * G{(R*D)^{ - 1}}$ such that $M \le W$ and $W = \left\{ {m{s^{ - 1}}|m \in M,s \in R * D} \right\}$.
	\end{proposition}

\begin{proof} 	It is easy to note that $R * G$ can be presented as a crossed product $(R * D) * ({G /D}) = R * G$ and hence any nonzero element of $R * G$ can be presented as a finite sum $\sum\nolimits_{g \in T} {{a_g}\bar g} $, where $0 \ne {a_g} \in R * D$ and  $T$ is a right transversal for $D$ in $G$. Since $R * D$ is a domain, it implies that all nonzero elements of the subring $R * D$ are regular in $R * G$. Then it follows from \cite[Theorem 2.1.12]{Pass1989} that the ring $R * G{(R*D)^{ - 1}}$ exists if $R * D$ satisfies the right Ore condition, i.e. for any nonzero elements $a \in R * D$ and  $x \in R * G$ there are nonzero elements $b \in R * D$ and $y \in R * G$ such that $ay = xb$. Let $x = \sum\nolimits_{g \in T} {{a_g}\bar g} $, as $R * D$ is an Ore domain, for each ${a_g}$ there is $0 \ne {b_g} \in R * D$ such that ${a_g} \cdot {b_g} = a \cdot {c_g}$ for some $0 \ne {c_g} \in R * D$. Since the Ore domain $R * D$ is uniform and ${({b_g}R * D)^{\bar g}}$ is a right ideal of $R * D$ for each $g \in G/A$, we see that there is a nonzero element $0 \ne b \in \bigcap\nolimits_{g \in T} {({b_g}R * D)^{\bar g}}$. Then for each $g \in T$ the element $b$ can be presented in the form $b = {({b_g}{d_g})^{\bar g}}$, where   
$ {d_g}\in R * D $ and hence $xb = (\sum\nolimits_{g \in T} {{a_g}\bar g} )b = $$\sum\nolimits_{g \in T} {{a_g}\bar g} {({b_g}{d_g})^{\bar g}} = $$\sum\nolimits_{g \in T} {{a_g}} ({b_g}{d_g})\bar g = \sum\nolimits_{g \in T} a {c_g}{d_g}\bar g = a(\sum\nolimits_{g \in T} {{c_g}} {d_g}\bar g) = ay$, where $y = \sum\nolimits_{g \in T} {{c_g}} {d_g}\bar g$. Thus, $R * D$ satisfies the right Ore condition and hence the ring $R * G{(R*D)^{ - 1}}$ exists. \par 
	(i) Since $R*D$ is an Ore domain, it is not difficult to show that a subset $R * D{(R*D)^{ - 1}} = \left\{ {r \cdot {s^{ - 1}}|r,s \in R * D} \right\}$ of $R * G{(R*D)^{ - 1}}$ is a subring of $R * G{(R*D)^{ - 1}}$ and $R * D{(R*D)^{ - 1}}$ is a division ring.  It follows from the definition of crossed products that $R * G$ can be presented as a crossed product $(R * D) * (G/D) = { \oplus _{t \in T}}(R * D)\bar t$. Therefore, $R * G{(R*D)^{ - 1}}={ \oplus _{t \in T}}(R * D{(R*D)^{ - 1}})\bar t = (R * D{(R*D)^{ - 1}}){ \otimes _{R * D}}R * G$. \par 
	
	(ii) Since, by (i),  $R * G{(R*D)^{ - 1}} = (R * D{(R*D)^{ - 1}}) * ({G /D})$ , where $R * D{(R*D)^{ - 1}}$ is   a division ring, the assertion follows from \cite[Proposition 1.6]{Pass1989}. \par 
	
(iii)  Let $I = An{n_{R * G}}(a)$, then $M \cong {{R * G}/I}$. Put $Q = R * G{(R*D)^{ - 1}}$ and $W = {Q /{IQ}}$ then $W$ is a cyclic $Q$-module. To prove the embedding $M \le W$ it is sufficient to show that $IQ\bigcap {R * G = I} $. Let $x \in IQ\bigcap {R * G} $ then $x = \sum\nolimits_{i = 1}^n {{b_i}a_i^{ - 1}} $, where ${b_i} \in I$ and ${a_i} \in R * D$. Since the Ore domain $R * D$ is uniform, the intersection of the right ideals generated by the elements ${a_i}$ is nonzero. Therefore, there exists an element $0 \ne y \in R * D$ such that $y = {a_i}{x_i}$ for each $i$, where ${x_i} \in R * D$. Then $xy = \sum\nolimits_{i = 1}^n {{b_i}{x_i}}  \in I$ and hence, as the modulus $M$ $R * D$-torsion-free, we can conclude that $x \in I$. Therefore, $IQ\bigcap {R * G}  \subseteq I$ and the inverse inclusion is obvious. Thus, $M \le W$ and we can assume that $W = M \cdot R * G{(R*D)^{ - 1}} = M{(R*D)^{ - 1}} = \left\{ {m{s^{ - 1}}|m \in M,s \in R * D} \right\}$.
\end{proof}

By the Krull dimension of a ring $R$ we mean the Krull dimension of the right $R$-module ${R_R}$. If $M$ is an $R$-module then ${K_R}(M)$denotes the Krull dimension of $M$. (see \cite[Chap. 6]{McRo} for the definitions).

\begin{lemma}\label{Lemma 4.1.2} Let $D$ be a nilpotent group which has a series $A \le C \le D$, where $A$ is an isolated abelian normal subgroup of D and the subgroup $C$ is dense in $D$. Let $R$ be a commutative domain and suppose that the group ring $RA$ has a  $G$-invariant prime ideal $P$. Let $\tilde R * \tilde D = RD/PRD$ and $\tilde R * \tilde C = RC/PRC$, where $\tilde R = RA/P$, $\tilde D = D/A$ and $\tilde C = C/A$. Then:\par 
(i) if $\left| {D:C} \right| < \infty $ and $M$ is an $\tilde R * \tilde D$-module then ${K_{\tilde R * \tilde D}}(M) = {K_{\tilde R * \tilde C}}(M)$;\par
(ii) if the domain $R$ and the group $D$ are finitely generated then the ring $\tilde R * \tilde D = RD/PRD$ has Krull dimension $\rho  = \dim \tilde R + h(D)$ and is $\rho $-critical;\par
	(iii) for any nonzero element $0 \ne \tilde a \in \tilde R * \tilde D$ there is an element $0 \ne \tilde b \in \tilde R * \tilde D$ such that $0 \ne \tilde a\tilde b = \tilde c \in \tilde R * \tilde C$;\par
	(iv) there exists a right ring of quotients $\tilde R * \tilde D{(\tilde R * \tilde D)^{ - 1}}$ = $\tilde R * \tilde C{(\tilde R * \tilde C)^{ - 1}}{ \otimes _{\tilde R * \tilde C{{(\tilde R * \tilde C)}^{ - 1}}}}\tilde R * \tilde D = { \oplus _{t \in T}}(\tilde R * \tilde C{(\tilde R * \tilde C)^{ - 1}})\bar t$, where $T$ is a right transversal for $C$ in $D$.
\end{lemma}

	\begin{proof} (i) Since any subgroup of the nilpotent group $D$ is subnormal, it is sufficient to consider the case where $C$ is a normal subgroup of $D$. In this case we can repeat the arguments of the proof of \cite[Lemma 8]{Sega1977}. \par
(ii). Since the group $D$ is finitely generated nilpotent, the subgroup $A$ is also finitely generated. Then, as the ring $R$ is finitely generated, we see that $\tilde R$ is a finitely generated commutative domain. Therefore, $\tilde R$ is a commutative Noetherian ring and hence, by \cite[Corollary 6.4.8]{McRo}, $\tilde R$ has Krull dimension ${K_{\tilde R}}(\tilde R) = \dim \tilde R = \mu $. Since $0$ is a prime ideal of $\tilde R$, we see that ${K_{\tilde R}}(\tilde R/I) = \dim (\tilde R/I) < \dim (\tilde R) = {K_{\tilde R}}(\tilde R) = \mu $ for any nonzero ideal $I$ of $\tilde R$. Therefore, $\tilde R$ is $\mu $-critical.\par
	Since the subgroup $A$ is isolated, the group $\tilde D$ is finitely generate nilpotent and torsion-free. Therefore, $\tilde D$ has a finite series each of whose factor is infinite cyclic. Then it is not difficult to show that the ring $\tilde R * \tilde D$ has a series of subrings $\tilde R = {S_0} \le {S_1} \le ... \le {S_n} = \tilde R * \tilde D$ such that ${S_i} = {S_{i - 1}} * \left\langle {{g_i}} \right\rangle $, where $\left\langle {{g_i}} \right\rangle $ is an infinite cyclic group and $1 \le i \le n$. It follows from \cite[Corollary A]{Rose1973} that any simple ${S_i}$-module is finite and hence it is an Artenian ${S_{i - 1}}$-module. The proof is by induction on $n$. If $n = 0$ then $.{S_0} = \tilde R$ is $\mu $-critical. Suppose that $.{S_{n - 1}}$ is $\rho $-critical then, by \cite[Lemma 10]{Sega1977}, $ {S_n}$ is $\rho  + 1$-critical and the assertion follows. \par
(iii) Let $a$ be a preimage of $\tilde a$ in $RD$ then the element $a$ can be presented in the form $a = \sum\nolimits_{i = 1}^n {{r_i}} {d_i}$, where ${r_i} \in R$ and ${d_i} \in D$. Let ${R_1}$ be a subring of $R$ generated by elements $\{ {r_i}|1 \le i \le n\} $ and ${D_1}$ be a subgroup of $D$ generated by elements $\{ {d_i}|1 \le i \le n\} $. Put ${C_1} = {D_1} \cap C$, ${A_1} = {D_1} \cap A$,  ${P_1} = {R_1}{A_1} \cap P$, ${\tilde C_1} = {C_1}/{A_1}$, ${\tilde D_1} = {D_1}/{A_1}$, and ${\tilde R_1} = {R_1}{A_1}/{P_1}$. By Lemma 2.2.3, ${P_1}{R_1}{D_1} = ({R_1}{A_1} \cap P){R_1}{D_1} = {R_1}{D_1} \cap PRD$ and hence $({R_1}{D_1} + PRD)/PRD \cong {R_1}{D_1}/({R_1}{D_1} \cap PRD) = {R_1}{D_1}/{P_1}{R_1}{D_1}$. Then, as $a \in {R_1}{D_1}$, we can assume that $\tilde a \in {R_1}{D_1}/{P_1}{R_1}{D_1} = {\tilde R_1} * {\tilde D_1}$ and hence it is sufficient to show that for any nonzero element $\tilde a \in {\tilde R_1} * \tilde D$ there is an element $0 \ne \tilde b \in {\tilde R_1} * \tilde D$ such that $0 \ne \tilde a\tilde b = \tilde c \in {\tilde R_1} * {\tilde C_1}$. Thus, to simplify the denotations, we can assume that the domain $R$ and the group $D$ are finitely generated. Then, as $C$ is a dense subgroup of the finitely generated nilpotent group $D$, it is not difficult to show that $\left| {D:C} \right| < \infty $ and hence $\left| {\tilde D:\tilde C} \right| < \infty $.\par
Let $0 \ne \tilde a \in \tilde R * \tilde D$ then it is sufficient to show that $\tilde a\tilde R * \tilde D \cap \tilde R * \tilde C \ne 0$. Suppose that $\tilde a\tilde R * \tilde D \cap \tilde R * \tilde C = 0$. Since $h(\tilde D) = h(\tilde C)$, it follows from (ii) that ${K_{\tilde R * \tilde D}}(\tilde R * \tilde D) = \dim \tilde R + h(D) = \dim \tilde R + h(C) = {K_{\tilde R * \tilde C}}(\tilde R * \tilde C) = \rho $, besides, the rings $\tilde R * \tilde D$ and $\tilde R * \tilde C$ are $\rho $-critical. Suppose that $\tilde a\tilde R * \tilde D \cap \tilde R * \tilde C = 0$ then $\tilde R * \tilde D/\tilde a\tilde R * \tilde D$ contains an $\tilde R * \tilde C$-submodule which is isomorphic to $\tilde R * \tilde C$ and hence ${K_{\tilde R * \tilde C}}(\tilde R * \tilde D/\tilde a\tilde R * \tilde D) = \rho $. On the other hand, as the ring $\tilde R * \tilde D$ is $\rho $-critical, we see that ${K_{\tilde R * \tilde D}}(\tilde R * \tilde D/\tilde a\tilde R * \tilde D) < \rho $. However, by (i), ${K_{\tilde R * \tilde C}}(\tilde R * \tilde D/\tilde a\tilde R * \tilde D) = {K_{\tilde R * \tilde D}}(\tilde R * \tilde D/\tilde a\tilde R * \tilde D)$ and the obtained contradiction shows that $\tilde a\tilde R * \tilde D \cap \tilde R * \tilde C \ne 0$. \par
	(iv) Since any subgroup of the nilpotent group $D$ is subnormal, there is a series $C = {D_0} \le {D_1} \le ... \le {D_n} = D$ and, by \cite[Corollary 37.11]{Pass1989}, each crossed product $\tilde R * {\tilde D_i}$ is an Ore domain, where ${\tilde D_i} = {D_i}/A$. Then, by Proposition 4.1.1(i), there is a right ring of quotients $\tilde R * \tilde D{(\tilde R * \tilde D)^{ - 1}}$ and it is not difficult to show that the right partial ring of quotients $\tilde R * \tilde D{(\tilde R * {\tilde D_{n - 1}})^{ - 1}}$ is a subring of $\tilde R * \tilde D{(\tilde R * \tilde D)^{ - 1}}$. Each element $x \in \tilde R * \tilde D{(\tilde R * \tilde D)^{ - 1}}$ can be presented in the form $x = a{b^{ - 1}}$, where $a,b \in \tilde R * \tilde D$. By (iii), there is an element $0 \ne c \in \tilde R * \tilde D$ such that $0 \ne bc = d \in \tilde R * \tilde C$. Then $bc{d^{ - 1}} = 1$ and hence ${b^{ - 1}} = c{d^{ - 1}}$.Thus, each element  $x \in \tilde R * \tilde D{(\tilde R * \tilde D)^{ - 1}}$ can be presented in the form $x = ac{d^{ - 1}}$, where $ac \in \tilde R * \tilde D$ and $d \in \tilde R * {\tilde D_{n - 1}}$, and we can conclude that $\tilde R * \tilde D{(\tilde R * \tilde D)^{ - 1}} = \tilde R * \tilde D{(\tilde R * {\tilde D_{n - 1}})^{ - 1}} = \tilde R * {\tilde D_{n - 1}}{(\tilde R * {\tilde D_{n - 1}})^{ - 1}}{ \otimes _{\tilde R * {{\tilde D}_{n - 1}}}}\tilde R * \tilde D$. By the inductive hypothesis $\tilde R * {\tilde D_{n - 1}}{(\tilde R * {\tilde D_{n - 1}})^{ - 1}} = \tilde R * \tilde C{(\tilde R * \tilde C)^{ - 1}}{ \otimes _{\tilde R * \tilde C}}\tilde R * {\tilde D_{n - 1}}$ and the arguments of \cite[Chap. 2, Lemma 2.1]{Karp} shows that $\tilde R * \tilde D{(\tilde R * \tilde D)^{ - 1}} = \tilde R * \tilde C{(\tilde R * \tilde C)^{ - 1}}{ \otimes _{\tilde R * \tilde C}}\tilde R * \tilde D$.\par
\end{proof}

	\begin{lemma}\label{Lemma 4.1.3} Let $G$ be a group which has a series $A \le C \le D \le G$  such that $D$ and $A$ are normal subgroups, the subgroup $D$ is nilpotent, the subgroup $A$ is abelian, the subgroup $C$ is dense in $D$, the quotient group $D/A$ is torsion-free and the quotient group $G/D$ polycyclic-by-finite. Let $R$ be a commutative domain and suppose that the group ring $RA$ has a $G$-invariant prime ideal $P$. Let $\tilde R * \tilde G = RG/PRG$, $\tilde R * \tilde D = RD/PRD$ and $\tilde R * \tilde C = RC/PRC$, where $\tilde R = RA/P$, $\tilde G = G/A$,$\tilde D = D/A$ and $\tilde C = C/A$. Then there exists a Noetherian partial right ring of quotients $\tilde R * \tilde G{(\tilde R * \tilde D)^{ - 1}}$ = $\tilde R * \tilde C{(\tilde R * \tilde C)^{ - 1}}{ \otimes _{\tilde R * \tilde C}}\tilde R * \tilde G = { \oplus _{t \in T}}(\tilde R * \tilde C{(\tilde R * \tilde C)^{ - 1}})\bar t$, where $T$ is a right transversal for $\tilde C$ in $\tilde G$.
\end{lemma}

	\begin{proof} By Proposition 4.1.1(i, ii), there exists a Noetherian partial right ring of quotients $\tilde R * \tilde G{(\tilde R * \tilde D)^{ - 1}}$ and by Lemma 4.1.2(iv), $\tilde R * \tilde G{(\tilde R * \tilde D)^{ - 1}}$ = $\tilde R * \tilde C{(\tilde R * \tilde C)^{ - 1}}{ \otimes _{\tilde R * \tilde C}}\tilde R * \tilde G = { \oplus _{t \in T}}(\tilde R * \tilde C{(\tilde R * \tilde C)^{ - 1}})\bar t$, where $T$ is a right transversal for $C$ in $G$.   
\end{proof}

	\begin{lemma}\label{Lemma 4.1.4} Let $G$ be a group, $R$ be a ring and $R * G$ be a crossed product. Let $Q$ be a ring  such that $R * G$ is a subring of $Q$ and there is a subring $U$ of $Q$ such that $R \le U$ and $Q = U{ \otimes _{R * H}}R * G = { \oplus _{t \in T}}Ut$, where $H$ is a subgroup of  $G$ such that $R * H = R * G \cap U$ and $T$ is a right transversal for $H$ in $G$. Let $M$ be a $Q$-module, let $0 \ne a \in M$ and let $J = an{n_Q}(a)$. Then $aQ = aU{ \otimes _{R * H}}R * G = { \oplus _{t \in T}}aUt$,where $T$ is a right transversal for $H$ in $G$, if and only if $J = (J \cap U)R * G = (an{n_U}(a))R * G$. 
\end{lemma}

          \begin{proof} Suppose that $aQ = aU{ \otimes _{R * H}}R * G = { \oplus _{t \in T}}aUt$, where $T$ is a right transversal for $H$ in $G$, and let $x \in an{n_Q}(a)$. Then we can present the element $x$ in the form $x = \sum\nolimits_{i = 1}^n {{x_i}} {t_i}$, where ${x_i} \in U$ and ${t_i} \in T$ for some . Since $ax = \sum\nolimits_{i = 1}^n {a{x_i}} {t_i} = 0$ and the sum $aQ = { \oplus _{t \in T}}aUt$ is direct, we can conclude that $a{x_i} = 0$ for each $i$. Therefore, $J \subseteq (J \cap U)R * G$ and, evidently, $(J \cap U)R * G \subseteq J$. So, we can conclude that $J = (J \cap U)R * G = (an{n_U}(a))R * G$. \par
Suppose now that $J = (J \cap U)R * G$ and the sum $aQ = \sum\nolimits_{t \in T} {aUt} $ is not direct, where $T$ is a right transversal for $H$ in $G$. Then there are ${x_i} \in U$ and ${t_i} \in T$ such that $a{x_i}{t_i} \ne 0$ for each $i$ but $\sum\nolimits_{i = 1}^n {a{x_i}} {t_i} = 0$ for some . Therefore, $\sum\nolimits_{i = 1}^n {{x_i}} {t_i} \in J$ and, as $J = (J \cap U)R * G$, we can conclude that each ${x_i} \in J \cap U = an{n_U}(a)$. It implies that $a{x_i}{t_i} = 0$ for each $i$ but it leads to a contradiction. So, we can conclude that $aQ = aU{ \otimes _{R * H}}R * G = { \oplus _{t \in T}}aUt$.  
\end{proof}

\begin{theorem}\label{Theorem 4.1.5} Let $G$ be a nilpotent $FART$-group and let $D$ be a normal subgroup of $G$ such that the quotient group $G/D$ is polycyclic. Let $k$ be a finitely generated field such that $char \, k \notin Sp(G)$ and let $M$ be a faithful $kG$-module. Suppose that the subgroup $D$ contains an isolated in $D$ abelian $G$-invariant subgroup $A$ such that  $P = Ann_{kA}(M)$ is a maximal $G$-invariant faithful ideal of $kA$. 
If the module $M$ is ${{kD}  /{PkD}}$-torsion-free then for any nonzero element $0 \ne a \in M$ there is a finitely generated subgroup $H \le G$ such that $akG = akH{ \otimes _{kH}}kG$. 
\end{theorem}

\begin{proof} Let $\tilde R = {{kA} /P}$, $\tilde G = {G / A}$ and $\bar D = {D /A}$. Then ${{kG} /{PkG}} = \tilde R * \tilde G$ is a crossed product of a commutative domain $\tilde R$ and a nilpotent group $\tilde G$. Evidently, ${{kD} /{PkD}} = \tilde R * \tilde D$ is a subring of $\tilde R * \tilde G$. Since $MP = 0$, we can consider $M$ as an $\tilde R * \tilde G$-module. By Lemma 4.1.3, there exists a Noetherian partial right ring of quotients $\tilde R * \tilde G{(\tilde R * \tilde D)^{ - 1}}$ and , by Proposition 4.1.1(iii), there exists a cyclic $R * G{(R*D)^{ - 1}}$-module $W = aR * G{(R*D)^{ - 1}}$ such that $aR * G \le W$ and $W = \left\{ {m{s^{ - 1}}|m \in aR * G,s \in R * D} \right\}$. \par
Since the ring $\tilde R * \tilde G{(\tilde R * \tilde D)^{ - 1}}$ is Noetherian, the annihilator $J = an{n_{\tilde R * \tilde G{{(\tilde R * \tilde D)}^{ - 1}}}}(a)$ is a finitely generated right ideal of $Q$. Let $X = \{ {x_i}|1 \le i \le n\} $ be a finite set of generators of $J = an{n_Q}(a)$ then each ${x_i}$ can be presented in the form ${x_i} = {d_i}{({c_i})^{ - 1}}$, where ${d_i} \in \tilde R * \tilde G$ and ${c_i} \in \tilde R * \tilde D$. Let $Y =  \cup _{i = 1}^n({\mathop{\rm supp}\nolimits} ({d_i}) \cup {\mathop{\rm supp}\nolimits} ({c_i}))$ then, as $G$ is a nilpotent $FART$-group, we can choose a finitely generated dense subgroup $\tilde L$ of $ \tilde G$ such that $Y \subseteq \ \tilde L$. It easily implies that ${d_i} \in \tilde R * \tilde L$ and ${c_i} \in \tilde R * \tilde C$ for each $i$, where $\tilde C = \tilde D \cap \tilde L$ and hence $X \subseteq R * \tilde L{(R * \tilde C)^{ - 1}}$. It follows from Lemma 4.1.3. that $\tilde R * \tilde G{(\tilde R * \tilde D)^{ - 1}}$= ${ \oplus _{t \in {T_C}}}(\tilde R * \tilde C{(\tilde R * \tilde C)^{ - 1}})\bar t$, where ${T_C}$ is a right transversal for $\tilde C$ in $\tilde G$ and it easily implies that $\tilde R * \tilde G{(\tilde R * \tilde D)^{ - 1}}$=${ \oplus _{t \in {T_L}}}(\tilde R * \tilde L{(\tilde R * \tilde C)^{ - 1}})\bar t$, where ${T_L}$ is a right transversal for $\tilde L$ in $\tilde G$  because, by Proposition 4.1.1(i), $\tilde R * \tilde L{(\tilde R * \tilde C)^{ - 1}} = { \oplus _{t \in \tilde L \cap T_C}}(\tilde R * \tilde C{(\tilde R * \tilde C)^{ - 1}})\bar t$. Therefore, $\tilde R * \tilde G{(\tilde R * \tilde D)^{ - 1}}=(\tilde R * \tilde L{(\tilde R * \tilde C)^{ - 1}})\tilde R * \tilde G$ and , as $X \subseteq \tilde R * \tilde L{(\tilde R * \tilde C)^{ - 1}}$, we can conclude that $J = an{n_Q}(a) = (J \cap (\tilde R * \tilde L{(\tilde R * C)^{ - 1}}))\tilde R * \tilde G$. \par
Then it follows from Lemma 4.1.4 that $a\tilde R * \tilde G{(\tilde R * \tilde D)^{ - 1}} = a\tilde R * \tilde L{(\tilde R * \tilde C)^{ - 1}}{ \otimes _{\tilde R * \tilde L}}\tilde R * \tilde G = { \oplus _{t \in {T_{\tilde L}}}}a\tilde R * \tilde L{(\tilde R * \tilde C)^{ - 1}}t$, where ${T_{\tilde L}}$ is a right transversal for $\tilde L$ in $\tilde G$. Since $a\tilde R * \tilde G \le a\tilde R * \tilde G{(\tilde R * \tilde D)^{ - 1}}$ and $a\tilde R * \tilde L \le a\tilde R * \tilde L{(\tilde R * \tilde C)^{ - 1}}$, we can conclude that $a\tilde R * \tilde G = a\tilde R * \tilde L{ \otimes _{\tilde R * \tilde L}}\tilde R * \tilde G = { \oplus _{t \in T}}a\tilde R * \tilde Lt$. Therefore, $akG = akL{ \otimes _{kL}}kG = { \oplus _{t \in T}}akLt$, where $L$ is the preimage of $\tilde L$ in $G$ and $T$ is a right transversal for $L$ in $G$. \par
	Let $I = an{n_{kL}}(a)$ then $P \le I$ and $\tilde I = an{n_{\tilde R * \tilde L}}(a)$, where $\tilde I = I/P$. Since the ideal $P$ is maximal, $\tilde R$ is a field and, as the nilpotent subgroup $\tilde L$ is finitely generated, $\tilde L$ is polycyclic. Then it follows from \cite[Proposition 1.6]{Pass1989} that the ring $\tilde R * \tilde L$ is Noetherian and hence the right ideal $\bar I$ is finitely generated. By Theorem 3.2.1(ii) , the ideal $P$ is finitely generated and, as $\tilde I = I/P$, we can conclude that the ideal $I$ is also finitely generated. It implies that there is a finitely generated subgroup $H$ of $L$ such that $I = (I \cap kH)kG$. Then it follows from Lemma 4.1.4 that $akL = akH{ \otimes _{kH}}kL$ and, as $akG = akL{ \otimes _{kL}}kG$, it follows from \cite[Chap. 2, Lemma 2.1]{Karp} that $akG = akH{ \otimes _{kH}}kG$. 
\end{proof}

\subsection{A set of commutative invariants for modules over group rings of finitely generated nilpotent groups} \label{Subsection 4.2}. 

Let $S$ be a commutative ring, $P$ be a prime ideal of $S$ and ${S_P}$ be the localization of the ring $S$ at the ideal $P$. Let $W$ be an $S$ -module then the support  $Supp(W)$ of the module $W$ consists of prime ideals $P \le S$ such that ${W_P} = W{ \otimes _S}{S_P} \ne 0$ (see \cite[Chap. II, \S 4.4]{Bour}). By \cite[Chap. IV, \S 1.4]{Bour}, if $S$ and $W$ are Noetherian then the set $\mu \,(W)$ of minimal elements of $Supp(W)$ is finite and consists of prime ideal of the ring $S$ which are minimal over $An{n_R}(W)$. \par
Let $N$ and $K$ be normal subgroups of a group $G$ such that $K \le N$ and the factor-group ${N /K}$ is free abelian with free generators $K{x_1},K{x_2},...,K{x_n}$. We will call the system of subgroups $\chi \, = \left\{ { < K,{{\left\{ {{x_j}} \right\}}_{j \in J}} > \left| {\,J \subseteq \left\{ {1,...,n} \right\}} \right.} \right\}$ a complete system of subgroups over $K$. \par
Let $R$ be a ring and let $I$ be a $G$-invariant ideal of the group ring $RK$. The ideal $I$ is said to be $G$-grand 
if $R /(R \bigcap I ) = f$ is a field, $\left| K / I^{\dag}  \right| < \infty $ and $I = (RF\bigcap I )RK$, where $F$ 
is a normal subgroup of $G$, such that ${I^\dag } \le F \le K$ and the quotient group ${F \{I^\dag }$ is abelian. 
\par
Since $\left| K/I^{\dag } \right| < \infty $ and ${I^\dag } \le K \le N$, we see that $N /I^{\dag }$ is a finitely generated finite-by-abelian group and hence $N /I^{\dag }$ contains a characteristic central subgroup $A$ of finite index. As the quotient group ${N /K}$ is finitely generated free abelian, the subgroup $A$ is also finitely generated free abelian. Then it follows from \cite[Proposition 1.6]{Pass1989} that the group ring $fA$ is Noetherian. \par
Let $W$ be a finitely generated $RN$ -module then $W / {WI}$ can be considered as an $fA$-module. Since $A$ is a subgroup of the finite index in $N /{I^\dag }$, we see that ${W /{WI}}$ is a finitely generated $fA$-module. Therefore, ${W /{WI}}$ is a Noetherian $fA$-module. \par

Let $\Omega $ be a finite set of generators of $W$ and let $V$ be an $RK$-submodule of $W$ generated by $\Omega $. Then, evidently, $W$ is an image of $V{ \otimes _{RK}}RN$ under some $RN$ -homomorphism $\alpha $. \par 
Put $\chi (W) = \{ X \in \chi{\, |\,} {Ker\,\alpha \bigcap (V{ \otimes _{RK}}RX) = 0} \}$,  
$r(W) = \max \{ r(X /K){\,| \,}X \in \chi (W)\}$ 
 and $\mu \chi (W) = \{ X \in \chi (W){\,| \,}r(X /K) = r(W)\}$. \par
Let ${\rho _A}({W /{WI}})$ be the set of prime ideals $P \in \mu ({W / {WI}})$ of maximal dimension and denote by $d({W / {WI}})$ this dimension if  $\mu(W /{WI}) \ne \emptyset $, i.e. ${W/ {WI}} \ne 0$. If ${W /{WI}} = 0$ then we put $d({W /{WI}}) =  - 1$. It is not difficult to note that $d({W / {WI}})$ does not depend on the choice of the central subgroup $A$ of finite index in $N /{I^\dag }$.

\begin{proposition}\label{Proposition 4.2.1} Let $N$ be a nilpotent group of finite rank which has a finitely generated normal subgroup $H$ such that $N = ZH$, where $Z$ is the centre of $N$. Let  $K$ be a normal subgroup of $H$ such that the quotient group ${H / K}$ is free abelian and $Z \cap H$ is an isolated subgroup of $K$ and let $\chi $ be a complete system of subgroups of $H$ over $K$. Let $R$ be a finitely generated commutative domain and $k$ be the field of fractions of $R$. Let $U$ be a cyclic $kN$-module such that $An{n_{kZ}}(U) = P$ is a prime ideal and $U = akN = akH{ \otimes _{kH}}kN$. Let $W = aRH$ and suppose that for any subgroup $X \in \mu \,\chi (W)$ the module $U$ is $k(XZ) /Pk(XZ)$-torsion-free. Then there is an $H$-grand ideal $I$ of the ring $RK$ such that , $d(W /WI) = d(V /VI) =  r(W)$ for any $RH$-submodule $0 \ne V \le W$ and: \par
(i) the module $W$ contains a submodule $\hat W \ne 0$ such that ${\rho _A}({\hat W / {\hat WI}})={\rho _A}({bRH / {bRHI}})$ for any nonzero element $0 \ne b \in \hat WkN = \hat W{ \otimes _{RH}}kN$. \par
(ii) if $char \, (kZ/P) = 0$ and $\pi $ a finite set of prime integers then we can choose the ideal $I$ such that $char \, (R/(R \cap I)) \notin \pi $.
\end{proposition}

\begin{proof} 
It follows from \cite[Proposition 1(i)]{Tush2002} that there is an $H$-grand ideal $I$ of the ring $RK$ such that , $d(W / WI) = d(V / VI) = $ $r(W)$ for any $RH$-submodule $0 \ne V \le W$. \par
(i) It follows from \cite[Proposition 1(ii)]{Tush2002} the module $W$ contains a submodule $\hat W \ne 0$ 
such that  ${\rho _A}({\hat W / {\hat W I}})={\rho _A}({\hat V / \hat V I})$  
for any submodule $0 \ne \hat V \le \hat W$.  Let $0 \ne b \in \hat WkN = \hat W{ \otimes _{RH}}kN$, as $N = ZH$ , the element $b$ can be presented in the form $b = \sum\nolimits_{i = 1}^n {{b_i}{\alpha _i}{z_i}} $, where ${b_i} \in \hat W$, ${\alpha _i} \in k$, ${z_i} \in T$ and $T$ is a right transversal for $Z \cap H$ in $Z$. Since $\hat WkN = { \oplus _{z \in T}}\hat WkHz$ and $T$ is a right transversal for $Z \cap H$ in $Z$, we see that $bRH \le  \oplus _{i = 1}^n{b_i}{\alpha _i}{z_i}RH$ and, as ${b_i} \in \hat W$, ${\alpha _i} \in k$ and ${z_i} \in T \subseteq Z$, it is not difficult to show that ${b_i}{\alpha _i}{z_i}RH \cong {V_i} \le \hat W$. Therefore, we can assume that $bRH \le  \oplus _{i = 1}^n{(\hat W)_i}$ and it easily implies that $bRH$ contains a nonzero submodule $0 \ne V$ which is isomorphic to some submodule of $\hat W$. Then it follows from \cite[Lemma 5(ii)]{Tush2002} that $d(V/VI) \le d(bRH/bRHI) \le d( \oplus _{i = 1}^n{(\hat W/\hat WI)_i}) = d(\hat W/\hat WI)$ and, as $d(V/VI) = d(\hat W/\hat WI)$, we can conclude that $d(V/VI) = d(bRH/bRHI) = d(\hat W/\hat WI)$. Then it follows from \cite[Lemma 5(iii)]{Tush2002} that $\rho _A    (V/VI) \subseteq \rho _A    (bRH/bRHI) \subseteq \rho _A    (\hat W/\hat WI)$ and hence, as $\rho _A    (V/VI) = \rho _A    (\hat W/\hat WI)$, we can conclude that $\rho _A    (bRH/bRHI) = \rho _A    (\hat W/\hat WI)$. \par
(ii) The assertion follows from \cite[Lemma 6]{Tush2002}.     
\end{proof}
     \begin{proposition}\label{Proposition 4.2.2} Let $H$ be a finitely generated nilpotent group and $K$ be a normal subgroup of $H$ such that the quotient group $H / K$ is free. Let $R$ be a commutative domain, $W$ be a finitely generated $RH$-modulus and $I$ be an $H$ -grand ideal  of $RK$ such that $f = R / (R \bigcap I )$ is a field of  characteristic $p$. Suppose that $W = V{ \otimes _{RL}}RH$ for some subgroup $L$ of the finite index in $H$ such that $K \le L$ and $p$ does not divide $\left| {H:L} \right|$. Let $A$  be a central subgroup of the finite index of $H / I^{\dagger} $. Let $C$ be  the subgroup of $A$ generated by controllers of ideals from ${\rho _A}(W /WI)$ and $D$ be the preimage of $C$ in $H$, then $D \le L$. 
\end{proposition}

     \begin{proof} The assertion was proved in \cite[Proposition 2]{Tush2002}.
\end{proof}

\subsection{Stabilised uniform modules over group rings of nilpotent FATR-groups} \label{Subsection 4.3}

	Let $R$ be a group and $W$ be an $R$-module. A submodule $U$ of $W$ is said to be essential if $U \cap X \ne 0$ for any nonzero submodule $X$ of $W$. The module $W$ is said to be uniform if each nonzero submodule of $W$ is essential. $R$-modules $W$ and $V$ are said to be similar if their injective hulls $[W]$ and $[V]$ are isomorphic. The modules $W$ and $V$ are similar if and only if they have isomorphic essential submodules. We will say that a submodule $U$ of $W$ is solid in $W$ if $W$ has no nonzero submodule which is isomorphic to a submodule of a proper quotient of  $U$. \par
	Let $G$ be a group, $H$ be a normal subgroup of $G$ and $W$ be an $RH$-module. By  \cite[Lemma 3.2(i)]{BrooBrow1985}, the stabilizer $Stab_{G}[W] = \{ g \in G\,|\, Wg \, and \, W \, are \, similar \} $ of the module $W$ in $G$ is a subgroup of $G$. If the module $W$ is uniform then $Stab_{G}[W] = Stab_{G}[V]$ for any nonzero submodule $V \le W$ because $[W] = [V]$. 

\begin{lemma}\label{Lemma 4.3.1} Let $G$ be a group and $N$ be a normal subgroup of $G$. Let $R$ be a   commutative domain and $W$ be a solid $RN$-submodule of an $RG$-module. Then: \par
          (i)  $WRG = WRS{ \otimes _{RS}}kG = { \oplus _{t \in T}}WRSt$, where $S = Sta{b_G}[W]$ and $T$ is a   right transversal for $S$ in $G$; \par
          (ii) if $WRG$ is an irreducible $RG$-module then $WRS$ is an irreducible $RS$-module; \par
          (iii) if $WRG$ is a semiprimitive $RG$-module then $WRS$ is an semiprimitive $RS$-module; \par
(iv) if the subgroup $N$ is abelian and $W \cong RN/P$, where $P$ is a  prime ideal of  $RN$ then $Sta{b_G}[W] = {N_G}(P)$ is the normaliser of $P$ in $G$.  
\end{lemma}

	\begin{proof} (i) It is sufficient to show that that every nonzero $RS$-submodule of $WRS$ contains a uniform $RN$-module. Then the assertion follows from \cite[Lemma 3.2(ii)]{BBrow1985}. \par
	Let $\tilde T$ be a right transversal for $N$ in $S$ then each nonzero element $b$ of $WRS$ can be presented in the form $b = \sum\nolimits_{i = 1}^n {{w_i}} {t_i}$, where $0 \ne {w_i} \in W$ and ${t_i} \in \tilde T$. We show by induction on $n$ that $bRS$ has a uniform $RN$-submodule. Since $b{t_1}^{ - 1} \in bRS$ we can assume that ${t_1} = 1$. If $n = 1$ then ${w_1}RN$ is a uniform $RN$-module. Suppose that there is an element $a \in RN$ such that ${w_1}a = 0$ and $ba \ne 0$ then $ba = \sum\nolimits_{i = 2}^n {{w_i}a} {t_i}$ and we can apply the inductive hypothesis. Thus, we can assume that $an{n_{RN}}(b) \ge an{n_{RN}}({w_1})$ and, as the module $W$ is solid in $WRG$, we can conclude that $an{n_{RN}}(b) = an{n_{RN}}({w_1})$. \par
       (ii) If $WRS$ has a proper submodule $U$ then $URG = U{ \otimes _{RS}}RG = { \oplus _{t \in T}}Ut$ is a   proper submodule of $WRG$ but it is impossible because the module $WRG$ is irreducible and the assertion follows. \par
           (iii) Suppose that there are a subgroup $K$ of $S$ and a $RK$-submodule $U$ of $WRS$ such that $WRS = U{ \otimes _{RK}}RS$. Then it follows from \cite[Chap. 2, Lemma 2.1]{Karp} that $WRG = U{ \otimes _{RK}}RG$ and, as the module $WRG$ is semiprimitive, the subgroup $K$ has finite index in $G$. Since $K \le S \le G$, it implies that the subgroup $K$ has finite index in $S$ and the assertion follows. \par
	(iv) Since for any $0 \ne w \in W$ and any $g \in G$ we have $an{n_{RN}}(w) = P$ and  $an{n_{RN}}(wg) = {P^g}$, we see that $W$ and $Wg$ have nonzero isomorphic submodules if and only if $P = {P^g}$. Then, as the module $W$ is uniform, we can conclude that $g \in Sta{b_G}[W]$ if and only if $g \in {N_G}(P)$. 
\end{proof}

\begin{lemma}\label{Lemma 4.3.2} Let $N$ be a nilpotent FATR-group of such that the quotient group $N/Z$ is finitely generated and torsion free, where $Z$ is the centre of $N$. Let $C = i{s_N}(ZN')$, where $N'$ is the derived subgroup of $N$. Let $k$ be a finitely generated field such that $char \, k \notin Sp(N)$ and $U$ be a uniform faithful $kN$-module such that $an{n_{kZ}}(U) = P$ is a maximal ideal of $kZ$. Suppose that the module $U$ is $kN/PkN$-torsion but for any subgroup $Y \le N$ such that $C \le Y$ and  $\left| {N:Y} \right| = \infty $ the module $U$ is $kY/PkY$-torsion-free. Then there exist a nonzero element $a \in U$ and a finitely generated dense subgroup $H$ of $N$ such that $N = ZH$, $char \, k \notin \pi (N/H)$, $akN = akH{ \otimes _{kH}}kN$ and for any nonzero element $b \in akN$ there is an element $c \in kN$ such that $0 \ne bc \in akH$.
\end{lemma}

\begin{proof} Let $V$ be a nonzero submodule of $U$ then it is easy to note that the submodule $V$ is $kN/PkN$-torsion but for any subgroup $Y \le N$ such that $C \le Y$ and  $\left| {N:Y} \right| = \infty $ the module $V$ is $kY/PkY$-torsion-free. Therefore, in the proof we can replace the module $U$ with any of its non-zero submodules. So, we can assume that $U = akN$ is a cyclic $kN$-module. \par
It follows from the definition of $C$ that the quotient group ${N / C}$ is free abelian. Then there is a subgroup $Y$ of $N$ such that $C \le Y$ and the quotient group $N/Y$ is infinite cyclic. The quotient ring $kN/PkN$ is a crossed product $\tilde k * \tilde N$, where $\tilde k = kZ/P$ and $\tilde N = N/Z$. Since $an{n_{kZ}}(U) = P$, we can consider $U$ as a torsion $\tilde k * \tilde N$-module which is $\tilde k * \tilde Y$-torsion-free and $U = a\tilde k * \tilde N$. By \cite[Corollary 37.11]{Pass1989}, $\tilde k * \tilde Y$ is an Ore domain and hence, by Proposition 4.1.1, there exist a partial ring of quotients $\tilde k * \tilde N{(\tilde k * \tilde Y)^{ - 1}}$ and a cyclic $\tilde k * \tilde N{(\tilde k * \tilde Y)^{ - 1}}$-module $a\tilde k * \tilde N{(\tilde k * \tilde Y)^{ - 1}}$ such that $U = a\tilde k * \tilde N \le a\tilde k * \tilde N{(\tilde k * \tilde Y)^{ - 1}}$. By Proposition 4.1.1(i), the ring $\tilde k * \tilde N{(\tilde k * \tilde Y)^{ - 1}}$ is a crossed product $Q * \left\langle {\tilde g} \right\rangle $, where $Q = \tilde k * \tilde Y{(\tilde k * \tilde Y)^{ - 1}}$ is a division ring and $\left\langle {\tilde g} \right\rangle  = \tilde N/\tilde Y$ is an  infinite cyclic group. Since the module $U$ is $\tilde k * \tilde N$-torsion and  $\tilde k * \tilde N = (\tilde k * \tilde Y) * \left\langle {\tilde g} \right\rangle $, it is not difficult to note that ${\dim _Q}aQ * \left\langle g \right\rangle  < \infty $. It easily implies that $aQ * \left\langle g \right\rangle $ contains an irreducible $Q * \left\langle g \right\rangle $-submodule $abQ * \left\langle g \right\rangle $ and, evidently, we can assume that $b \in \tilde k * \tilde N$. So, replacing $a$ with $ab$ we can assume that the module $a\tilde k * \tilde N{(\tilde k * \tilde Y)^{ - 1}}$ is irreducible. Then for any $0 \ne u \in U = a\tilde k * \tilde N$ there is an element $\tilde v{\tilde s^{ - 1}} \in \tilde k * \tilde N{(\tilde k * \tilde Y)^{ - 1}}$ such that $u\tilde v{\tilde s^{ - 1}} = a$ and hence $u\tilde v = a\tilde s \ne 0$, where $\tilde v \in \tilde k * \tilde N$ and $\tilde s \in \tilde k * \tilde Y$. \par
By Theorem 4.1.5, there is a finitely generated subgroup $H \le N$ such that $akN = akH{ \otimes _{kH}}kN$. 
Since the quotient group $N/Z$ is finitely generated, if necessary we can take the finitely generated subgroup $H$  bigger such that $N = HZ$. It easily implies that $H$ is a normal subgroup of $N$.  Since $char \, k \notin Sp(N)$, the $p$-component ${N_p}/H$ of $N/H$ is finite if $p = char \, k > 0$.  Then replacing $H$ with ${N_p}$ we can assume that $char \, k \notin \pi (N/H)$. \par
Since for any $0 \ne u \in U = a\tilde k * \tilde N$ there are elements $\tilde v \in \tilde k * \tilde N$ and $\tilde s \in \tilde k * \tilde Y$ such that $u\tilde v = a\tilde s \ne 0$, it follows from Lemma 4.1.2. that there is an element $0 \ne \tilde b \in \tilde R * \tilde Y$ such that $0 \ne \tilde s\tilde b = \tilde c \in \tilde k * \widetilde {(Y \cap H)}$ and hence $0 \ne u\tilde v\tilde b = a\tilde s\tilde b \in akH$. 
\end{proof}

\begin{proposition}\label{Proposition 4.3.3} Let $G$ be a group and $N$ be a normal nilpotent FATR-subgroup of the group $G$ such that the quotient group $N/Z$ is finitely generated and torsion-free, where $Z$ is the centre of $N$. Let $C = i{s_N}(ZN')$, where $N'$ is the derived subgroup of $N$. Let $k$ be a finitely generated field such that $char \, k \notin Sp(N)$ and $U$ be a uniform faithful $kN$-module such that $an{n_{kZ}}(U) = P$ is a maximal G- invariant ideal of $kZ$. Suppose that the module $U$ is $kN/PkN$-torsion but for any subgroup $Y \le N$ such that $C \le Y$ and  $\left| {N:Y} \right| = \infty $ the module $U$ is $kY/PkY$-torsion-free. If $Sta{b_G}[U] = G$ then $N$ contains a $G$-invariant finitely generated dense subgroup.   
\end{proposition}

\begin{proof} Let $V$ be a nonzero submodule of $U$ then it is easy to note that the submodule $V$ is $kN/PkN$-torsion but for any subgroup $Y \le N$ such that $C \le Y$ and  $\left| {N:Y} \right| = \infty $ the module $V$ is $kY/PkY$-torsion-free. Since module $U$ is uniform, $Stab_{G}[U] = Stab_{G}[V]$ and therefore in the proof we can replace the module $U$ with any of its non-zero submodules. So, we can assume that $U = akN$ is a cyclic $kN$-module. Then, by Lemma 4.3.2, we can assume that there exists a finitely generated dense subgroup $H$ of $N$ such that $N = ZH$, $char \, k \notin \pi (N/H)$, $akN = akH{ \otimes _{kH}}kN$ and for any nonzero element $b \in akN$ there is an element $c \in kN$ such that $0 \ne bc \in akH$. \par
 Let $L = \bigcap\nolimits_{g \in G} {{H^g}} $ then $L$ is a $G$-invariant subgroup of $H$. Since $H'=(ZH)'=  N'$ is a $G$-invariant subgroup of $H$, we can conclude that $H' \le L$ and hence the quotient group $H/L$ is abelian. \par
	Suppose that the quotient group $H/L$ is infinite. Then, by Lemma 2.1.3, there are a countable subset $X=\{g_i \in G\, |\, i \in \mathbb{N}\}$ and a descending chain $\{H_i \, |\, i \in \mathbb{N}\}$ of $H$-invariant subgroups ${H_i} \le H$ of finite index in $H$ which satisfies the conditions (i-iii) of Lemma 2.1.3 and such that the quotient group $H/(\bigcap_{ i \in \mathbb{N}} H_i)$ is free abelian. 	Put  $K = \bigcap_{ i \in \mathbb{N}} H_i$ and $S = ZK$, it is not difficult to show that  
	$C \le S$ and the module $U$ is $kS/PkS$-torsion-free. \par
Let $R \le k$ be a finitely generated domain such that $k$ is the field of fractions of $R$. Then it is not difficult to check that an $RH$-module $W = aRH$ satisfies all conditions of Proposition 4.2.1. Therefore,  there are an $H$-grand ideal $I$ of the group ring $RK$ and a central characteristic subgroup $A$ of finite index in $ K/I^\dagger $ such that the module $W$ contains a submodule $\hat W \ne 0$ such that $ \rho _A   (bRH/bRHI)= \rho _A   (\hat W/\hat W  I) =  \rho _A $   for any nonzero element $0 \ne b \in \hat WkN = \hat W{ \otimes _{RH}}kN$. Then replacing the element $a$ with a nonzero element of $\hat W$ we can assume that $ \rho _A   (bRH/bRHI)= \rho _A   ( aRH/aRHI) =  \rho _A   $ for any nonzero element $0 \ne b \in akN = aRH{ \otimes _{RH}}kN$.  \par 

Let $f = R/(R \cap I)$, then it follows from the definition of the $H$-grand ideal $I$ that $f$ is a finite field. If $char\, k > 0$ then $char\, k = char\, f$ and, as $char\, k \notin \pi (N/H)$, we see that $char\, f \notin \pi (N/H)$. By Lemma 2.1.3(iii), $\pi (H/{H_i}\,) \subseteq \pi (N/H\,)$ and hence $char\, f \notin \pi (H/{H_i}\,)$ for any 
$ i \in \mathbb{N} $. If $char\, k = 0$ and, by Lemma 2.1.3(i), there is a prime integer $q \in \pi (N/H)$ such that $H/{H_i}$ is a $q$-group for all   
$ i \in \mathbb{N} $   then, by Proposition 4.2.1(ii), we can choose the ideal $I$ such that $char\, f \ne q$. Thus, according to Lemma 2.1.3(i), we can assume that there is $ n \in \mathbb{N} $ such that $char \, f \notin \pi ({H_n}/{H_i})$ for all $i \ge n$. \par

Since $Stab_{G}[U] = G$, for any $g \in G$ we have $[U]g \cong [U]$and it easily implies that there are nonzero elements $b, d \in U$ such that $(bkG)g \cong dkG$. By the choice of the subgroup $H$, there is an element $c \in kN$ such that $0 \ne {a_g} = bc \in akH$ and hence  we can choose the element $c $ such that $ a_g \in  RH $. 
Therefore, as $akN = aRH{ \otimes _{RH}}kN$, we have ${a_g}kN = {a_g}RH{ \otimes _{RH}}kN$. Since $({a_g}kN)g = ({a_g}g)kN$, we can conclude that $({a_g}g)kN = ({a_g}g)k{H^g}{ \otimes _{R{H^g}}}kN$. Then, as $(bkN)g \cong dkN$, there is an element ${d_g} \in dkN$ such that $({a_g}g)kN \cong {d_g}kN$. Therefore, $({a_g}g)RH \cong {d_g}RH$ and, as $({a_g}g)kN = ({a_g}g)R{H^g}{ \otimes _{R{H^g}}}kN$, it is not difficult to note that $({a_g}g)RH = ({a_g}g)R(H \cap {H^g}){ \otimes _{R(H \cap {H^g})}} \cong {d_g}RH$. \par
Let ${S_g}$ be a subgroup of $H$ such that $H \cap {H^g} \le {S_g}$ and ${S_g}/(H \cap {H^g})$ is the Sylow $p$-subgroup of ${S_g}/(H \cap {H^g})$, where $char \, f = p$. Then $({a_g}g)RH = ({a_g}g)R{S_g}{ \otimes _{R{S_g}}}$ and $p$ does not divide $\left| {H/{S_g}} \right|$. Let $C_A$ be a subgroup of $A$ generated by controllers of ideals from ${\rho _A}({W / {WI}})$ and $D$ be the preimage of $C_A$ in $H$. Then, as  and $({a_g}g)RH \cong {d_g}RH$, it follows from Proposition 4.2.2 that $D \le {S_g}$. \par

	Thus, if ${L_i} = H \cap {H^{{g_i}}}$, where $g_i \in X  $ and ${S_i} = {S_{{g_i}}}$ then,  
	as descending chain  $\{H_i \, |\, i \in \mathbb{N}\}$  satisfies the condition (ii) of Lemma 2.1.3, we have $\bigcap\nolimits_{j = 1}^i {{L_j}}  \subseteq {H_i}$ and $D \le {S_i}$ for each $i \in \mathbb{N}$. Then it follows   from Lemma 2.1.4 that 	$D \leq K = \bigcap_{ i \in \mathbb{N}} H_i$. It implies that $C_A = 1$ but it is impossible   because ${\rho _A}({\hat W / {\hat WI}})$ consists of proper ideals of $fA$. The obtained contradiction shows that the   quotient group $H/L$ is finite and hence $L$ is a $G$-invariant finitely generated dense subgroup of $N$.   
\end{proof}

\section{On irreducible representations of nilpotent FATR-groups of nilpotency class 2} \label{Section 5}

  Primitive irreducible modules are a basic subject for investigations when we are dealing with induced modules and, naturally, the following question appears: what can be said on the construction of a group $G$ if it admits a faithful primitive (or semiprimitive) irreducible representation over a field $k$? It should be noted that there are many results which show that the existence of a faithful irreducible representation of a group $G$ over a field $k$ may have essential influence on the structure of the group $G$ (see for instance \cite{Szec2016, SzTu2017, Tush1990,Tush1993,Tush2012}). 
\par   
	In \cite{Harp1977} Harper solved a problem raised by Zaleskii and proved that any not abelian-by-finite finitely generated nilpotent group has an irreducible primitive representation over a not locally finite field. In \cite{Tush2002}  we proved that if a minimax nilpotent group $G$ of nilpotency class  2 has a faithful irreducible primitive representation over a finitely generated field of characteristic zero then the group $G$ is finitely generated. In \cite{Harp1980} Harper studied polycyclic groups which have faithful irreducible representations. It is well known that any polycyclic group is finitely generated soluble of finite rank and meets the maximal condition for subgroups (in particular, for normal subgroups). In \cite{Tush2000} we showed that in the class of soluble groups of finite rank with the maximal condition for normal subgroups only polycyclic groups may have faithful irreducible primitive representations over a field of characteristic zero. 
\par   
If $G$ is a group then the $FC$-center $\Delta (G) = \{ g \in G|\left| {G:{C_G}(g)} \right| < \infty \} $of $G$ is a characteristic subgroup of $G$. In \cite[Theorem A]{Harp1980} Harper proved that if a polycyclic group $G$ has a faithful primitive irreducible representation over a field $k$ then $\Delta (G)$ is rather large in the sense that $\Delta (G) \cap H > 1$ for any subgroup $1 \ne H$ of $G$ such that $|G:{N_G}(H)| < \infty $.  
\par   
By Auslander theorem, any polycyclic group $G$ is linear over the field  and it is well known that the group $G$ is finitely generated of finite rank. In \cite{Tush2022-1} we study finitely generated linear (over a field of characteristic zero) groups of finite rank which have faithful irreducible primitive representations over a field of characteristic zero. We prove that if an infinite finitely generated linear group $G$ of finite rank has a faithful irreducible primitive representation over a field of characteristic zero then $\Delta (G)$ is infinite (see \cite[Theorem 6.1]{Tush2022-1}). \par

\subsection{Semiprimitive irreducible representations of nilpotent FATR-groups of nilpotency class 2} \label{Subsection 5.1}

\begin{lemma}\label{Lemma 5.1.1} Let $G$ be a nilpotent group of nilpotency class 2 such that the torsion subgroup $T$ of $G$ is contained in the centre $Z$ of $G$ and let ${G_1}$ be a subgroup of finite index in $G$. Let $k$ be a field and let $M$ be a faithful $kG$-module such that $M = {M_1}kG$, where and ${M_1}$ is a $k{G_1}$-submodule of $M$. Then ${M_1}$ is a faithful $k{G_1}$-module.
\end{lemma}

\begin{proof} Suppose that $X = {C_{{G_1}}}({M_1}) \ne 1$ then $X$ is an ${G_1}$-invariant subgroup of $G$. Since $Y = Z \cap X$ is a normal subgroup of $G$ and $M = {M_1}kG$, we see that $Y \le {C_G}(M)$. However, the module $M$ is faithful (i.e. ${C_G}(M) = 1$) and hence $X \cap Z = 1$. Then, as $T \le Z$, the subgroup $X$ is torsion-free. So, as ${G_1} \le {N_G}(X)$ and ${G_1}$ has finite index in $G$, it follows from Lemma 2.1.2(iii) that there are  and a normal subgroup $\bar X$ of $G$ such that ${X^n} \le \bar X \le X$ and ${X^n} \ne 1$ because the subgroup $X$ is torsion-free. Since $\bar X$ centralizes ${M_1}$ and $\bar X$ is a normal subgroup of $G$, it is not difficult to note that $\bar X$ centralizes $M = {M_1}kG$ but it is impossible because the $kG$-module $M$ is faithful. The obtained contradiction shows that the $kN$-module ${M_1}$ is faithful. 
\end{proof}

\begin{lemma}\label{Lemma 5.1.2} Let $G$ be a group, let $B$ and $X$ be normal subgroups of $G$ such that the subgroup $B$ is abelian, $B \cap X = Z$ is the centre of $G$, the quotient group $B/Z$ is finite and the quotient group $X/Z$ is free abelian of finite rank. Let $k$ be a field and let $M$ be a $kG$-module such that $an{n_{kB}}(M) = P$ is a maximal $G$-invariant ideal of $kB$ and the module $M$ is $kW/PkW$-torsion, where $W = B \cdot X$. Then the module $M$ is $kX/{P_1}kX$-torsion, where ${P_1} = ann_{kZ}(M) = P \cap kZ$. 
\end{lemma}

\begin{proof} Since $ann_{kB}(M) = P$ is a maximal ideal of $kB$ and, by Lemma 2.2.1(ii), ${P_1} = an{n_{kZ}}(M) = P \cap kZ$ is a maximal ideal of $kZ$, we see that ${F_1} = kZ/{P_1}$ is a subfield of the field $F = kZ/P$ and, as $|B/Z| < \infty $, we see that ${\dim _{{F_1}}}F < \infty $ . It is not difficult to note that $\bar W = W/B \cong X/Z = \bar X$ are free abelian groups and hence, by \cite[Corollary 37.11]{Pass1989}, the crossed products $F * \bar W = kW/PkW$and ${F_1} * \bar X = kX/{P_1}kX$ are Ore domains and ${S_1} = {F_1} * \bar X$ is a subring of $S = F * \bar W$. Since $P$ annihilates the module $M$, we can consider $M$ as an $S$-module. \par
Thus, it is sufficient to show that any $S$-torsion $S$-module $M$ is also ${S_1}$-torsion. The proof is by induction on $r(X/Z)$. Suppose that $r(X/Z) = 1$ and let $0 \ne a \in M$. Then  and, as $an{n_S}(a) \ne 0$, we can conclude that ${\dim _F}aS < \infty $. Since ${\dim _{{F_1}}}F < \infty $, we see that ${\dim _{{F_1}}}aS < \infty $ and hence elements $a{g^n}$, where , are linearly dependent over ${F_1}$. Therefore, there are elements ${\gamma _1},...,{\gamma _m} \in {F_1}$ such that $a{\gamma _1}{g^{{n_1}}} + ... + a{\gamma _m}{g^{{n_m}}} = a({\gamma _1}{g^{{n_1}}} + ... + {\gamma _m}{g^{{n_m}}}) = 0$. Thus, $an{n_{{S_1}}}(a) \ne 0$ for any $0 \ne a \in M$. \par
 	Suppose that $r(X/Z) = n > 1$ and the assertion holds if $r(X/Z) < n$. Evidently, there is a subgroup $Y$ of $X$ such that $Z \le Y < X$ and the quotient group $X/Y$ is infinite cyclic and hence $X = Y \cdot \left\langle g \right\rangle $. It easily implies that $W = BY \cdot \left\langle g \right\rangle $. Therefore, $S = R * \left\langle g \right\rangle $ and ${S_1} = {R_1} * \left\langle g \right\rangle $, where $R = kBY/P(kBY)$ and $R_1 = kY/{P_1}kY$. Since $R = F * \bar Y$ and ${R_1} = {F_1} * \bar Y$, where $\bar Y = Y/Z \cong YB/B$, it follows from \cite[Corollary 37.11]{Pass1989} that $R$ and ${R_1}$ are Ore domains. Besides, as $\left| {BY/Y} \right| = \left| {B/B \cap Y} \right| = \left| {B/Z} \right| < \infty $, it is not difficult to note that $R$ is a finitely generated ${R_1}$-module. \par
	Suppose that there is an element $0 \ne a \in M$ such that $ann_{S_1}(a) = 0$. Since the module $M$ is $S$-torsion, it follows from Lemma 2.2.4 that there is a finitely generated $R$-submodule $N$ of $aS$ such that the quotient module $aS/N$ is $R$-torsion. Then it follows from the inductive hypothesis that the quotient module $aS/N$ is ${R_1}$-torsion and, as $R$ is a finitely generated ${R_1}$-module, $N$ is a finitely generated ${R_1}$-module. Since the ring ${R_1}$ is Noetherian, so is the finitely generated ${R_1}$-submodule $N$. Therefore, ${N_1} = a{S_1} \cap N$ is a finitely generated ${R_1}$-submodule of $a{S_1}$ such that the quotient module $a{S_1}/{N_1}$  is ${R_1}$-torsion. Since $an{n_{{S_1}}}(a) = 0$, we have $ aS_1 \cong \oplus_{n \in \mathbb{Z}}{R_1}g^n $ and it easily implies that $a{S_1}$ has no finitely generated ${R_1}$-submodule ${N_1}$ which define the ${R_1}$-torsion quotient module $a{S_1}/{N_1}$. The obtained contradiction shows that $an{n_{{S_1}}}(a) \ne 0$ and hence the module $M$ is $kX/{P_1}kX$-torsion.
\end{proof}

\begin{lemma}\label{Lemma 5.1.3} Let $G$ be a group which has a series $A < D$ of normal subgroups such that the subgroup $A$ is abelian and the quotient group $D/A$ is torsion-free nilpotent. Let $k$ be a field and let $M$ be a faithful $kG$-module such that $P = an{n_{kA}}(M)$ is a $G$-invariant maximal ideal of $kA$. Then:\par
	(i) if there is an element $0 \ne a \in M$ such that $an{n_{kD}}(a) > P$ then the submodule $akG$ is $kD/PkD$-torsion;   \par
(ii) if the module $M$ is irreducible then the module $M$ is either $kD/PkD$-torsion-free or $kD/PkD$-torsion; \par
(iii) if the module $M$ is irreducible and $kD/PkD$-torsion, the centre $Z$ of $G$ is a dense subgroup of $A$ and the section $D/Z$ is central in $G$ then there is a normal subgroup $X$ of $G$ such that $Z < X \le D$, the quotient group $X/Z$ is finitely generated and the module $M$ is $kX/(P \cap kZ)kX$-torsion but for any subgroup $Y$ of $G$ such that $Z \le Y < X$ and $|X:Y| = \infty $ the module $M$ is $kY/(P \cap kZ)kY$-torsion-free. 
\end{lemma}

\begin{proof} (i). Suppose that the module $M$ is not $kD/PkD$-torsion-free then there is an element $0 \ne a \in M$ such that $an{n_{kD}}(a) > P$. Since ${kD} / {PkD}$ is a crossed product $R * \bar D$, where $R = kA/P$ is a field and $\bar D = D/A$, it follows from \cite[Corollary 37.11]{Pass1989} that ${kD} / {PkD} = R * \bar D$ is an Ore domain and hence the ring $R * \bar D$ is uniform. Evidently, as $P = ann_{kA}(M)$, we can consider $M$ as an $R * \bar G$-module and $R * \bar D \le R * \bar G$, where $\bar G = G/A$.  Each element $c$ of $M = akG$ can by presented in the form $c = ab$, where $b \in R * \bar G$. The element $b$ can be presented in the form $b = \sum\nolimits_{v \in V} {{\beta _v}} v$, where ${\beta _v} \in k$ and $V$ is a finite subset of the group $G$ . Since $ann_{kD}(a) > P$, we see that $ann_{R * \bar D}(a) \ne 0$ and, as the ring $R * \bar D$ is uniform, we can conclude that $J = \bigcap\nolimits_{v \in V} {{v^{ - 1}}(ann_{R * \bar D}(a))v}  \ne 0$. Then, as $c = ab = \sum\nolimits_{v \in V} {a{\beta _v}} v$, we have $cJ = 0$. Thus, for any element $c$ of $M$ we have $ann_{R * \bar D}(c) \ne 0$. \par
(ii). The assertion follows from (i) because $M = akG$ for any element $0 \ne a \in M$. \par
(iii). Since the module $M$ is $kD/PkD$-torsion, $PkD < an{n_{kD}}(a)$ for any nonzero element $a \in M$. Let $0 \ne b \in an{n_{kD}}(a)\backslash PkD$ and let $W = \left\langle {Z \cup Supp(b)} \right\rangle $ be a subgroup of $D$ generated by elements of $Z$ and $Supp(b)$ then $Z < W$ and the quotient group $W/Z$ is finitely generated abelian . Besides, as $0 \ne b \in an{n_{kD}}(a)\backslash PkD$, it easy to note that $Supp(b)$ is not contained in $A$ and hence the quotient group $W/Z$ is infinite. As $0 \ne b \in an{n_{kD}}(a)\backslash PkD$, we can conclude that $0 \ne b \in an{n_{kW}}(a)\backslash (P \cap kW)kW$ and hence the module $M$ is not $kW/(P \cap kW)kW$-torsion-free. Since the section $D/Z$ is central in $G$, any subgroup $V$ of $G$ such that $Z \le V \le D$ is $G$-invariant. The quotient group $W/Z$ is infinite finitely generated abelian and hence there are normal subgroups $B$ and $X$ such that $Z \le B$ $Z < X$, $\left| {B/Z} \right| < \infty $, the quotient group is finitely generated free abelian and $W/Z = B/Z \times X/Z$. Since the quotient group $A/Z$ is torsion and the quotient group $D/A$ is torsion-free, it is not difficult to show that $B = A \cap W$. As $0 \ne b \in an{n_{kW}}(a)\backslash (P \cap kW)kW$, we see that $an{n_{kW}}(a) > (P \cap kB)kW$. Then it follows from (ii) that the module $M$ is $kW/(P \cap kB)kW$-torsion and hence, by Lemma 5.1.2, the module $M$ is $kX/(P \cap kZ)kX$-torsion. Since the section $D/Z$ is central in $G$,  any subgroup $Y \le X$ such that $Z \le Y$ is $G$-invariant and hence, by (ii), the module $M$ is either $kY/(P \cap kZ)kY$-torsion or $kY/(P \cap kZ)kY$-torsion. It easily implies that we can choose the subgroup  $X$  such that  the rank of $X/Z$ is minimal with respect the property that the module $M$ is $kX/(P \cap kZ)kX$- torsion. Therefore, for any subgroup $Y$ of $G$ such that $Z \le Y < X$ and $|X:Y| = \infty $ the module $M$ is $kY/(P \cap kZ)kY$-torsion-free. 
\end{proof}

\begin{lemma}\label{Lemma 5.1.4} Let $k$ be a finitely generated field and let $G$ be a nilpotent FATR-group of nilpotency class 2 such that the torsion subgroup $T$ of  $G$ is contained in the centre $Z$ of $G$ and $char\,k \notin Sp(G)$. Let $M$ be an irreducible faithful semiprimitive $kG$-module. Suppose that the group $G$ is not finitely generated. Then: \par
       (i) the group $G$ has a normal series $Z \le A \le D \le G$ such that the quotient group ${G /D}$ is finitely generated, the quotient group ${D /A}$ is torsion-free and has no infinite polycyclic quotient groups, the subgroup $A$ is abelian and $A \le is_{G}(Z)$; \par
       (ii) there is an element $0 \ne a \in M$ such that $P = ann_{kA}(a)$ is a faithful maximal ideal of $kA$, the normaliser ${G_1} = {N_G}(P)$ of the ideal $P$ in the group $G$ is a subgroup of finite index in $G$, $M = {M_1}{ \otimes _{k{G_1}}}kG$, where ${M_1} = ak{G_1}$ is an irreducible faithful semiprimitive $k{G_1}$-module and $P = an{n_{kA}}({M_1})$; \par
        (iii)  ${D_1} = D \cap {G_1} > A$ and the module ${M_1}$ is ${{k{D_1}} \mathord{\left/
 {\vphantom {{k{D_1}} {Pk{D_1}}}} \right.
 \kern-\nulldelimiterspace} {Pk{D_1}}}$-torsion.
\end{lemma}

\begin{proof} (i). Let $X = i{s_G}(Z)$, it is not difficult to show that $X$ is a normal subgroup of the subgroup $G$. It follows from Lemma 2.1.2(ii) that the commutator map ${\varphi _a}:X \to Z$ given by ${\varphi _a}:x \mapsto [a,x]$ is a homomorphism for any $a \in X$. On the other hand, if $a \in X$ then there is a positive integer $n$ such that ${a^n} \in Z$ and hence ${\varphi _a}{(x)^n} = {[a,x]^n} = [{a^n},x] = 1$ for all $x \in G$. Therefore, for any $a \in X$ there is a positive integer $n$ such that ${\varphi _a}{(X)^n} = 1$. It implies that for any $a \in X$ the mapping ${\varphi _a}$ maps $X$ into $T$, i.e. ${\varphi _a}(X) \le T$. Let $\pi $ be the set of prime divisors of orders of elements of $T$. Since  $G$  is a FATR-group, it is not difficult to show that the set $\pi $ is finite, 
$r(T)<\infty$ and hence the subgroup $T$ is Chernikov. Then, as   the subgroup $T$ is Chernikov and ${\varphi _a}{(X)^n} = 1$, we can conclude that the subgroup ${\varphi _a}(X)$ is finite $\pi $-group. Let $A$ be the centre of the subgroup $X$, as $X$ is a normal subgroup of $G$, we see that $A$ is a normal subgroup of $G$. Evidently, $Z \le A \le {C_X}(a) = Ker{\varphi _a}$ for any $a \in X$ and $A = { \cap _{a \in X}}{C_X}(a)$. Then it follows from Remak theorem that $X/A \le { \times _{a \in X}}(X/{C_X}(a))$ and, as ${\varphi _a}(X) \cong X/Ker{\varphi _a} = X/{C_X}(a)$, we can conclude that $X/A \le { \times _{a \in X}}{\varphi _a}(X)$, where ${\varphi _a}(X)$ are finite $\pi $-groups. Since $Z \le A \le X = is_{G}(Z)$, we see that the quotient group $X/A$ is torsion and, as $X/A \le { \times _{a \in X}}{\varphi _a}(X)$, it is not difficult to show  that $X/A$  is a residually finite torsion $\pi $-group. Then, as $X/A$ has finite tank $r(X/A)$ and the set $\pi $ is finite, the quotient group $X/A$ is Chernikov and as it is residually finite, we see that $X/A$ is finite. \par
Thus, group $G$ has a normal series $Z \le A \le X \le G$ such that the quotient group ${X / A}$ is finite, the quotient group ${G / X}$ is torsion-free and the quotient group ${G / A}$ is abelian. Since $\left| {X/A} \right| = n < \infty $, it is not difficult to show that ${G^n}A \cap X = A$. Therefore, the quotient group ${G^n}A/A$ is torsion-free and, as the group $G$ has finite rank, the quotient group $G/{G^n}A$ is finite. Since the quotient group ${G^n}A/A$ is torsion-free abelian of finite rank, it is not difficult to show that there exists a subgroup $D$ of $G$ such that $A \le D \le {G^n}A$ the quotient group ${G / D}$ is finitely generated, the quotient group ${D / A}$ is torsion-free and does not have infinite polycyclic  quotient groups. \par
(ii). By \cite[Theorem 4.3(i)]{Brow1982}, $I = an{n_{kZ}}(M)$ is a maximal ideal of $kZ$ and, as module $M$ is faithful, we see that the ideal $I$ is faithful. Suppose that $Y = {(an{n_{kA}}(a))^\dag } \ne 1$ for some element $0 \ne a \in M$. Since ${I^\dag } = 1$, we see that $Y \cap Z = 1$ and, as $T \le Z$, we can conclude that the subgroup $Y$ is torsion-free. Then, as $A \le i{s_G}(Z)$, for any $1 \ne g \in Y$ there is  such that $1 \ne {g^n} \in Y \cap Z$ but this contradicts $Y \cap Z = 1$. Thus, ${(an{n_{kA}}(a))^\dag } = 1$ for any element $0 \ne a \in M$ and hence the $kA$-module $M$ is fully faithful. Then it follows from Theorem 3.2.3 that there is an element $0 \ne a \in M$ such that $P = an{n_{kA}}(a)$ is a faithful prime ideal of $kA$. It is easy to note that $P \cap kZ = an{n_{kZ}}(M)$ because $P = an{n_{kA}}(a) \ge an{n_{kZ}}(M)$ and $an{n_{kZ}}(M)$ is a maximal ideal of $kZ$. Then, by Lemma 2.2.1(ii), the ideal $P$ is maximal and hence$W = akA$ is solid in $M = akG$. Then it follows from Lemma 4.3.1(i,iv) that $M = akG = ak{G_1}{ \otimes _{k{G_1}}}kG = {M_1}{ \otimes _{k{G_1}}}kG$, where ${M_1} = ak{G_1}$ and ${G_1} = {N_G}(P)$. By Lemma 4.3.1(ii, iii), ${M_1}$ is an irreducible semiprimitive $k{G_1}$-module. Besides, as $P = an{n_{kA}}(a)$, ${M_1} = ak{G_1}$ and ${G_1} = {N_G}(P)$, we can conclude that $P = ann_{kA}({M_1})$. It follows from Lemma 5.1.1 that the $k{G_1}$-module ${M_1}$ is faithful. \par
          (iii). Since the group $G$ is not finitely generated and the subgroup ${G_1}$ has finite index in $G$, we see that the subgroup ${G_1}$ is not finitely generated. If the module ${M_1}$ is $k{D_1} / Pk{D_1}$-torsion-free then, by Theorem 4.1.5, for any element $0 \ne a \in {M_1}$ there is a finitely generated dense subgroup $H \le {G_1}$ such that $ak{G_1} = akH{ \otimes _{kH}}k{G_1}$. Since the group ${G_1}$ is not finitely generated, we see that the subgroup $H$ has infinite index in ${G_1}$ and this contradicts the semiprimitivity of the module  $M$.  Thus, the module ${M_1}$ is not $k{D_1} / Pk{D_1}$-torsion-free then, by Lemma 5.1.3(ii), the module ${M_1}$ is $k{D_1} / Pk{D_1}$-torsion. 
\end{proof}

\begin{theorem}\label{Theorem 5.1.5} Let $k$ be a finitely generated field and let $G$ be a nilpotent FATR-group of nilpotency class 2 such that the torsion subgroup $T$ of $G$ is contained in the centre $Z$ of $G$ and $char\,k \notin Sp(G)$. Suppose that the group $G$ admits a faithful semiprimitive irreducible representation $\varphi $ over the field $k$. Then the group $G$ is finitely generated. 
\end{theorem}

	\begin{proof} At first, we apply Lemma 5.1.4. Let $M$ be a $kG$-module of the representation $\varphi $. Replacing $G$ and $M$ with ${G_1}$ and ${M_1}$ from Lemma 5.1.4(ii) and applying Lemmas 5.1.1 and 4.3.1(ii,iii) we can assume that the group $G$ has a normal series $Z \le A < D \le G$ such that the quotient group ${G / D}$ is finitely generated, the quotient group ${D / A}$ is torsion-free and does not have infinite polycyclic quotient groups, the subgroup $A$ is abelian and $A \le i{s_G}(Z)$. Besides, $P = an{n_{kA}}(M)$ is a $G$-invariant maximal ideal of $kA$ and the model $M$ is $kD/PkD$-torsion. Then, by Lemma 5.1.3, there is a $G$-invariant subgroup $X$ of $D$ such that $Z < X$, the quotient group $X/Z$ is finitely generated free abelian, the module $M$ is $kX/{P_Z}kX$-torsion but for any subgroup $Y$ of $G$ such that $Z \le Y < X$ and $|X:Y| = \infty $ the module $M$ is $kY/{P_Z}kY$-torsion-free, where ${P_Z} = P \cap kZ$. \par
Since $kX/{P_Z}kX = \tilde k * \tilde X$, where $\tilde k = kZ/{P_Z}$ and $\tilde X = X/Z$, it follows from \cite[Proposition 1.6]{Pass1989} and \cite[Corollary 37.11]{Pass1989} that $kX/{P_Z}kX = \tilde k * \tilde X$ is a Noetherian domain. Then, by \cite[Lemma 6.2.3]{McRo},  for any $0 \ne a \in M$ the module $akX$ has Krull dimension and hence we can choose an element $0 \ne a \in M$ such that $akX$ has minimal possible for nonzero elements of $M$ Krull dimension $\rho $. It follows from \cite[Lemma 6.2.10]{McRo} that  we can choose the element $a$ such that the module $akX$ is $\rho $-critical and hence, by \cite[Lemma 6.2.12]{McRo}, the module $akX$ is uniform. Since the submodule $akX$ is $\rho $-critical and $\rho $ is minimal possible Krull dimension for nonzero submodules of $M$, it is not difficult to note that $akX$ is solid in $M$. Then, by Lemma 4.3.1(i), $M = akG = akS{ \otimes _{kS}}kG$, where $S = Sta{b_G}[akX]$, and hence, as the module $M$ is semiprimitive, we have $\left| {G:S} \right| < \infty $. By Lemma 5.1.1, $akS$ is a faithful $kS$-module and the arguments of Lemma 4.3.1(ii,iii) shows that the module $akS$ is irreducible and semiprimitive. Thus, changing $G$ by $S$ and  $M$ by $akS$ we can assume that $G = Sta{b_G}[akX]$. Then, by Proposition 4.3.3, $X$ contains a finitely generated dense $G$-invariant subgroup $L$. \par
It is not difficult to show that  $C = {C_G}(L)$ is a normal subgroup of $G$, $Z \le C$ and it follows from \cite[Theorem 3.2.4]{LeRo2004} that the quotient group $G/C$ is polycyclic. Since the quotient group $A/Z$ is torsion, the quotient group ${D / A}$ is torsion-free and has no infinite polycyclic quotient groups, we see that quotient group ${D / Z}$  has no   infinite polycyclic quotient groups. Then, as $Z \le C$ and the quotient group $G/C$ is polycyclic, we can conclude that $\left| {D/D \cap C} \right| < \infty $ and hence, as $L \le D$, we have $\left| {L/L \cap C} \right| < \infty $ . Since $C = {C_G}(L)$, we see that $U = L \cap C$is an abelian $G$-invariant subgroup of finite index in $L$. Then it is not difficult to show that $V = ZU$ is an abelian $G$-invariant subgroup of finite index in $X$. Then it follows from Lemma 4.1.2(iii) that the module $M$ is not  $kV/{P_Z}kV$-torsion-free and hence, by Lemma 5.1.3(ii), the module $M$ is $kV/{P_Z}kV$-torsion. Since $|X:V| < \infty $, we see that for any subgroup $Y$ of $G$ such that $Z \le Y < V$ and $|X:Y| = \infty $ the module $M$ is $kY/{P_Z}kY$-torsion-free. Thus, to avoid the new notations we can assume that $X = V$, i.e. the subgroup $X$ is abelian. \par
	Since the ring $kX/{P_Z}kX = \tilde k * \tilde X$ is Noetherian, there is a nonzero element $a \in M$ such that $I = An{n_{kX}}(a)$ is a prime ideal and by the above arguments we can assume that $akX$ is a solid submodule of $M$ such that $G = Sta{b_G}[akX]$. Then it follows from Lemma 4.3.1(iv) that $I$ is a $G$-invariant ideal of $kX$. Since the module $M$ is $kX/{P_Z}kX$-torsion, we also can conclude that ${P_Z}kX < I$. 
As it was shown above, $X$ has a dense $G$-invariant subgroup $L$. Put ${Z_L} = L \cap Z$ and ${P_L} = kL \cap P$ then it is easy to note that ${P_L} = k{Z_L} \cap {P_Z}$ is a maximal ideal of $k{Z_L}$. Since ${P_Z}kX < I$, it follows from Lemma 2.2.1(iii) that $kL \cap ({P_Z}kX) < kL \cap I$. Put ${I_L} = kL \cap I$, as ${P_L}kL < kL \cap ({P_Z}kX)$ and $kL \cap ({P_Z}kX) < kL \cap I$, we can conclude that ${P_L}kL < {I_L}$ and, as ${P_L}$ is a maximal ideal of $k{Z_L}$, we have ${P_L} = {I_L} \cap kZ_L^{}$. Thus, we have $({I_L} \cap k{Z_L})kL < {I_L}$.
	On the other hand, since ${I_L} = kL \cap I$ is a $G$-invariant prime ideal of $kL$, it follows from \cite[Theorem D]{Rose1978} that ${I_L} = ({I_L} \cap k{\Delta _G}(L))kL$, where ${\Delta _G}(L) = L \cap \Delta (G)$. By Lemma 2.1.2(iv), $\Delta (G) \le i{s_G}(Z)$ and, as the quotient group $L/{Z_L}$ is torsion-free, it easily implies that ${\Delta _G}(L) = L \cap \Delta (G) \le {Z_L}$. Therefore, as ${I_L} = ({I_L} \cap k{\Delta _G}(L))kL$, it is not difficult to show that ${I_L} = ({I_L} \cap k{Z_L})kL$ but it contradicts $({I_L} \cap k{Z_L})kL < {I_L}$. The obtained contradiction shows that the group $G$ is finitely generated.

\end{proof}

\subsection{Irreducible representations of nilpotent minimax groups of nilpotency class 2 over a finitely generated field} \label{Subsection 5.2}

\begin{lemma}\label{Lemma 5.2.1} Let $k$ be a field and let $G$ be a torsion free  nilpotent group of nilpotency class 2. Suppose that the group $G$ has a faithful irreducible representation $\varphi $ over the field $k$ which is induced from an irreducible representation $\psi $ of a subgroup $N$ of $G$. Then the torsion subgroup $t(N/Ker\psi )$ is contained in the centre $Z(N/Ker\psi )$ of the quotient group $N/Ker\psi $. 
\end{lemma}

\begin{proof} Let $W$ be a $kG$-module of the representation $\varphi $ and let $V$ be a $kN$-module of the representation $\psi $ then $W = V{ \otimes _{kN}}kG$. Let $C = Z(G) \cap Ker\psi $ then $C \le {C_G}(V)$. As $C \le Z(G)$ is a normal subgroup of $G$, ${C_W}(C)$ is a $kG$-submodule of $W$. However, as $V \le {C_W}(C) \ne 0$ and $W$ is an irreducible $kN$-module, we can conclude that ${C_W}(C) = W$. Then $C \le Ker\varphi $ and, as the representation $\varphi $ is faithful, we can conclude that $C = 1$. Thus, $Z(G) \cap Ker\psi  = 1$. \par
Suppose that there exists a not trivial element $\bar g$ of $t(N/Ker\psi )$ such that $\bar g \notin Z(N/Ker\psi )$. Let $g$ be the preimage of the element $\bar g$ in the subgroup $N$ then there exists a positive integer $n$ such that ${g^n} = h \in Ker\psi $. Let $a \in N$, as $Ker\psi $ is a normal subgroup of $N$, we have that $[h,a] \in Ker\psi $. On the other hand, as the group $N$ is of nilpotency class 2, we see that $[h,a] \in Z(N)$. Therefore, $[h,a] \in Z(N) \cap Ker\psi $ and, as $Z(N) \cap Ker\psi  = 1$, we can conclude that $[h,a] = 1$ for any element $a \in N$. It immediately implies that ${g^n} = h \in Z(N)$. However, as $N$ is a torsion-free nilpotent group, its centre $Z(N)$ is an isolated subgroup of $N$ and hence $g \in Z(N)$. Thus, we can conclude that $\bar g \in Z(N/Ker\psi )$. 
\end{proof}

\begin{theorem}\label{Theorem 5.2.2}  Let $G$ be a torsion free minimax nilpotent group of nilpotency class 2 and let $k$ be a finitely generated field such that $shark \notin Sp(G)$. Suppose that the group  $G$ admits a faithful irreducible representation $\varphi $ over $k$. Then there exist a subgroup $N$ and an irreducible representation $\psi $ of the subgroup $N$ such that the representation $\varphi $ is induced from $\psi $ and the quotient group $N/Ker\psi $ is finitely generated. 
\end{theorem}

\begin{proof} Let $W$ be a $kG$-module of the representation $\varphi $. It follows from results of \cite[Proposition 5.1.5]{LeRo2004} that the minimax group $G$ has no infinite descending chain of subgroups   such that $\left| {{G_i}:{G_{i + 1}}} \right| = \infty $  for all . Then  it easily follows from \cite[Chap. 2, Lemma 2.1]{Karp} that there are a subgroup $N$ of $G$ and a semiprimitive irreducible $kN$-submodule $V$ of $W$ such that $W = V{ \otimes _{kN}}G$. Let $\psi $ be the representation of $N$ over $k$ such that $V$ is a  $kN$-module of the representation $\psi $. Then it follows from Lemma 5.2.1 that the torsion subgroup $t(N/Ker\psi )$ is contained in the centre $Z(N/Ker\psi )$ of the quotient group $N/Ker\psi $. Therefore, by Theorem 5.1.5, the quotient group $N/Ker\psi $ is finitely generated.
\end{proof}

\begin{corollary}\label{Corollary 5.2.3} Let $G$ be a torsion-free minimax nilpotent group of nilpotency class 2 and let $k$ be a finitely generated field such that $shark = 0$. Suppose that the group  $G$ admits a faithful irreducible representation $\varphi $ over $k$. Then there exist a subgroup $N$ and an irreducible primitive representation $\psi $ of the subgroup $N$ over $k$ such that the representation $\varphi $ is induced from $\psi $ and the quotient group $N/Ker\psi $ is finitely generated. 
\end{corollary}

\begin{proof} The assertion follows from the above theorem and \cite[Theorem 5.6]{Tush2022}. 
\end{proof}

\end{document}